\documentclass[onefignum,onetabnum]{siamart190516}

\usepackage{lipsum}
\usepackage{amsfonts}
\usepackage{graphicx}
\usepackage{epstopdf}
\usepackage{algorithmic}
\ifpdf
  \DeclareGraphicsExtensions{.eps,.pdf,.png,.jpg}
\else
  \DeclareGraphicsExtensions{.eps}
\fi

\newsiamremark{hypothesis}{Hypothesis}
\crefname{hypothesis}{Hypothesis}{Hypotheses}
\newsiamthm{claim}{Claim}

\headers{Evolving Neural Network (ENN) Method}{Z. Cai and B. Hejnal}

\title{Evolving Neural Network (ENN) Method \\ for One-dimensional 
Scalar Hyperbolic Conservation Law{\small s}: \\[1mm]  I. Linear and Quadratic Fluxes\thanks{This work was supported in part by the National Science Foundation under grant DMS-2110571.}}

\author{Zhiqiang Cai\thanks{Department of Mathematics, Purdue University, 150 N. University Street, West Lafayette, IN 47907-2067 
  (\email{caiz@purdue.edu}, \email{bhejnal@purdue.edu}).}
\and Brooke Hejnal\footnotemark[2]
 }

\usepackage{amsopn}

\ifpdf
\hypersetup{
  pdftitle={ENN1},
  pdfauthor={}
}
\fi

\DeclareMathSymbol{\shortminus}{\mathbin}{AMSa}{"39}
\NewDocumentCommand{\grad}{e{_^}}{%
  \mathop{}\!
  \nabla
  \IfValueT{#1}{_{\mspace{-4mu}#1}}
  \IfValueT{#2}{^{#2}}
}

\usepackage[utf8]{inputenc}
\usepackage{bm}
\usepackage{mathtools}
\usepackage{indentfirst}
\usepackage{subfigure}
\usepackage{tcolorbox} 
\usepackage{color}
\usepackage{xcolor}
\usepackage[export]{adjustbox}
\usepackage{scalerel}
\usepackage{float}
\usepackage{enumitem}

\usepackage[normalem]{ulem}

\usepackage{algorithmic}
\Crefname{ALC@unique}{Line}{Lines}

\newcommand{\R}{\mathbb{R}}

\usepackage{graphicx}
\usepackage{diagbox}
\usepackage{pifont}
\newcommand{\vertiii}[1]{{\left\vert\kern-0.25ex\left\vert\kern-0.25ex\left\vert #1 
    \right\vert\kern-0.25ex\right\vert\kern-0.25ex\right\vert}}

\usepackage{tikz}
\usetikzlibrary{shapes.misc}
\usetikzlibrary{decorations.pathreplacing}

\usepackage{enumitem}
\setlist[itemize]{left=16pt} 

\def\bb{{\bf b}}

\def\bt{{\bf t}}

\def\cM{{\cal M}}

\begin{document}

\maketitle

\begin{abstract}
We propose and study the evolving neural network (ENN) method for solving one-dimensional scalar hyperbolic conservation laws with linear and quadratic spatial fluxes. The ENN method first represents the initial data and the inflow boundary data by neural networks. Then, it evolves the neural network representation of the initial data along the temporal direction. The evolution is computed using a combination of characteristic and finite volume methods. For the linear spatial flux, the method is not subject to any time step size, and 
it is shown theoretically that the error at any time step is bounded by the representation errors of the initial and boundary condition. For the quadratic flux, an error estimate is studied in a companion paper \cite{ENN2}.
Finally, numerical results for the linear advection equation and the inviscid Burgers equation are presented to show that the ENN method is more accurate and cost efficient than traditional mesh-based methods. 
\end{abstract}




\section{Introduction}
Nonlinear hyperbolic conservation laws are still computationally challenging despite the successful application of many efficient and accurate numerical methods to practical problems \cite{leveque1992numerical, GoRa:96, thomas2013numerical, hesthaven2017numerical}. There are two significant difficulties when approximating hyperbolic conservation laws. The first difficulty is that the solution of the underlying problem is often discontinuous due to a discontinuous initial condition, a discontinuous inflow boundary condition, or a shock formation. Moreover, the locations of the discontinuities are unknown {\it a priori}. The second difficulty is that the strong form of the partial differential equation is invalid where the solution is discontinuous.

Traditional numerical methods such as finite difference, finite volume, and finite element methods are based on partitions of the computational domain into uniform or quasi-uniform meshes. These methods require a sufficiently small mesh size to achieve an accurate approximation.
Furthermore, when the solution of the underlying problem is discontinuous, the mesh size must be small enough to sharply capture the discontinuities. Because computation on a fine uniform mesh can be costly, adaptive mesh refinement \cite{BurgerOliger84, shen11, hartmann02} may be used instead. However, the efficiency of adaptive mesh refinement depends heavily on {\it a  posteriori} error indicators, and there is scarce research on the development of {\it a  posteriori} error indicators for hyperbolic conservation laws \cite{burman2009posteriori, de2004least, houston1999posteriori, houston2000posteriori, liu2020adaptive}. Methods such as Arbitrary Lagrangian Eulerian (ALE) methods (see, e.g., a review paper \cite{donea04}) and moving mesh methods \cite{lucier86, holden88, miller81, herbst83} avoid fine uniform meshes by evolving the solution with a dynamic mesh. In particular, \cite{lucier86} presents a semi-discrete scheme using characteristics with the first proof of quadratic convergence. 

Traditionally, the invalidity of the strong form of the partial differential equation at discontinuities is addressed with an additional constraint, the Rankine-Hugoniot (RH) jump condition \cite{leveque1992numerical, GoRa:96}, which is needed at discontinuous interfaces. Typically, the RH jump condition is enforced numerically through a finite volume approach for conservative schemes such as Roe and ENO/WENO (see, e.g., \cite{roe1981approximate, harten1987uniformly, shu1988efficient, leveque1992numerical}). Additionally, semi-Lagrangian methods maintain a uniform grid while utilizing characteristics to evolve the solution (see, e.g, \cite{qiu11, huang16}).

Recently, neural networks (NN) have been used as a new class of approximating functions to approximate the solution to partial differential equations (see, e.g., \cite{Cai2020, Weinan18, raissi2019physics,Sirignano18}). A neural network function is a linear combination of compositions
of linear transformations and a nonlinear univariate activation function. 
As shown in \cite{Cai2021linear, Cai2023linear}, a ReLU neural network can better approximate discontinuous functions with unknown interfaces than the existing classes of approximating functions such as polynomials or piecewise polynomials. Furthermore, the space-time least-squares ReLU neural network (LSNN) method was used to solve the linear advection-reaction equation with a discontinuous solution in \cite{Cai2021linear, Cai2023linear} and nonlinear hyperbolic conservation laws in \cite{Cai2023nonlinear}.
Hence, neural networks can resolve the first difficulty for hyperbolic conservation laws by finding discontinuous interfaces and avoiding fine uniform meshes.

The purpose of this paper is to introduce a novel numerical method that is derived from the physics of conservation laws. An important physical phenomenon described by conservation laws is the transportation of an initial profile along characteristic curves.
Based on this observation, we propose the evolving neural network (ENN) method which first represents an initial profile using a neural network and then evolves the representation of the initial profile along the temporal direction. In one dimension, the initial data is a function of the spatial variable $x\in\R$, which can be approximated well by a one hidden-layer ReLU neural network or equivalently a free knot spline. The neural network approximation is uniquely determined by nonlinear parameters (breaking points, i.e., input bias) and linear parameters (output weights and bias). For linear transport problems, breaking points propagate along their characteristic lines, and the linear parameters are computed using the nodal values of the breaking points. 
For nonlinear hyperbolic conservation laws, breaking points propagate either along their characteristic lines or according to a characteristic finite volume scheme depending on their locations. 

For linear transport problems, we show that the error at any time step is bounded by the representation errors of the initial data and the boundary data. For hyperbolic conservation laws with a quadratic spatial flux, the error estimate is studied in a companion paper \cite{ENN2}.
Numerical results for benchmark test problems for the linear advection equation and Burgers' equation are presented to show that the ENN method is more accurate and efficient than traditional mesh-based methods. 

The paper is organized as follows. Section~\ref{prob-sec} introduces the one-dimensional hyperbolic conservation law. Sections~\ref{shallow-nn} and \ref{init_bdry_data} describe shallow ReLU neural networks and how they can approximate initial and boundary data. Then, the ENN method for linear problems is derived and analyzed in Section \ref{ENN-linear}. Section \ref{ENN-nonlinear} derives the ENN method for Burgers' equation, which introduces the characteristic finite volume scheme and the time step control procedure. Numerical results for benchmark test problems for both the linear advection and the inviscid Burgers equations
are presented in Section \ref{num-exp}. Finally, some conclusive remarks are made in Section \ref{conclusion}.


\section{Problem Formulation} \label{prob-sec}

Consider the one-dimensional scalar hyperbolic conservation law of the form
\begin{equation} \label{pde}
    \left\{\begin{array}{rcll}
    \dfrac{\partial }{\partial t}u(x,t)
    +\dfrac{\partial }{\partial x}f\big(u(x,t)\big) &=& 0, &\text{ in }\,\, \Omega  \times I = \R\times (0,T), \\[2mm]
    u(x,0) &=& u_0(x), &\text{ in }\,\, \Omega,
    \end{array}\right.
\end{equation}
where the initial data $u_0$ is a given scalar-valued function and $f(u)$ is the spatial flux. In this paper, we consider the linear advection equation and Burgers' equation which correspond to the respective 
linear and quadratic spatial flux functions 
\[
 f(u) =\alpha u \quad\mbox{and}\quad f(u)= \dfrac12 u^2.
\]
Computation is often done in a bounded domain 
$\Omega=(a,b)$. In this case, the following inflow boundary condition is needed as a supplement to (\ref{pde})
\begin{equation}\label{bc}
    u= g(t), \text{ on }\,\, \Gamma_{-}.
\end{equation}
Here, $\Gamma_-$ is the part of the boundary $\partial \Omega \times I$ where the characteristic curves enter the domain $\Omega  \times I$, and the boundary data $g(t)$ is a given scalar-valued function.

A characteristic curve $x(t)$ of equation (\ref{pde}) emanating from a point $(\hat{x},\hat{t})$ in $\Omega \times I$ satisfies the following ordinary differential equation
\begin{equation}\label{c-curve}
\left\{\begin{array}{rcll}
	\dfrac{dx(t)}{dt} &=& f^\prime\big(u(x(t), t)\big), & \text{ for } t\in (0,T),\\[4mm]
		x(\hat{t}) &=& \hat{x}. 
    \end{array}\right. 
\end{equation}
The derivative of the solution $u(x,t)$ of (\ref{pde}) along a characteristic curve $x(t)$ satisfies
\begin{equation}\label{u-x(t)}
\dfrac{d}{dt}u\big(x(t), t\big)= u_x \,\dfrac{dx(t)}{dt} + u_t = u_x \,f^\prime(u) + u_t=0, 
\end{equation}
where $u_x$ and $u_t$ are the partial derivatives of $u$ with respect to the spatial and temporal variables, $x$ and $t$, respectively. This implies that the solution $u(x,t)$ is a constant along a characteristic curve $x(t)$. Hence, 
\begin{equation}\label{constant}
    u\big(x(t),t\big)=u\big(\hat{x},\hat{t}\big), 
\end{equation}
which, together with (\ref{c-curve}), implies that before $x(t)$ intersects another characteristic line, it is given by the line
\begin{equation}\label{c-line}
    x(t) = x(\hat{t})+ (t-\hat{t})f^\prime\big(u\left(\hat{x},\hat{t}\right)\big).
\end{equation}


\section{One Dimensional Shallow ReLU Neural Network}\label{shallow-nn} 

In one dimension, a one hidden-layer shallow ReLU neural network with $n$ neurons generates the following set of functions
\[
\cM_n = \left\{ c_{-1}+\sum_{i=0}^n c_i \rho(\omega_i x - b_i) : \, b_i \in \mathbb{R}, \, c_i \in \mathbb{R}, \, \omega_i \in \mathbb{R} \right\},
\]
where $\{ \omega_i\}_{i=0}^n$ and $\{b_i\}_{i=0}^n$ are the weights and bias of the hidden layer and $\{c_i\}_{i=0}^n$ and $c_{-1}$ are the weights and bias of the output layer, respectively, and 
\[
\rho(x) = \max\{0, x\}=\left\{\begin{array}{cc}
  0,   & \mbox{if }\, x<0, \\[2mm]
  x,  &  \mbox{if }\, x\ge 0
\end{array}\right.
\]
is the ReLU activation function. The input weights and bias are nonlinear parameters, and the output weights and bias are linear parameters. As discussed in \cite{LiuCai1}, the $2(n+1)$ nonlinear parameters can be reduced by half without compensating the approximation power by choosing $\omega_i=1$ for all $i=0,1,\ldots,n$. Moreover, by setting $b_0=a$ and restricting the breaking points $b_i$ to the interval $(a,b)$ for $i=1,\ldots,n$, the resulting set becomes
\begin{equation}\label{func1}
\cM_n =\left\{c_{-1}+c_0(x-a)+\sum_{i=1}^nc_i\rho(x-b_i):\, b_i\in [a,b),\, c_i\in \R\right\}.
\end{equation}
 
Each neural network function $v(x;\bb)$ in $\cM_n$ is a continuous piecewise linear function with respect to its fixed partition of breaking points
\[
\Delta: \, a=b_0< b_1< \cdots < b_n< b_{n+1}=b,
\]
where $\bb=(b_1,\ldots,b_{n})$. Define the global basis functions with respect to the partition $\Delta$ as $\rho_{-1}(x)=\rho_{-1}(x;b_0)=1$ and
\[
\rho_i(x)=\rho_i(x;b_i)=\rho(x-b_i)\big|_{(a,b)}=\left\{\begin{array}{ll}
x-b_i, & x\in (b_i,b),\\[2mm]
0, & x\notin (b_i,b),
\end{array}\right. 
\]
for $i=0,1,...,n$. Then, $v(x;\bb)$ is represented as a linear combination of the global basis functions $\big\{\rho_i(x;b_i)\big\}_{i=-1}^n$
\begin{equation}\label{g-rep}
    v(x;\bb)=\sum_{i=-1}^nc_i\rho_i(x;b_i).
\end{equation}

A key step in the evolving neural network (ENN) method proposed in this paper (see Section~\ref{ENN-linear}) involves propagating breaking points along characteristic lines, where the nodal values remain unchanged (see (\ref{constant})). Hence, the solution at a given time $t$ can be more easily represented using nodal basis functions. To this end, let $h_i=b_{i+1}-b_i$ be the distance between each pair of breaking points for $i=0,1,\ldots, n$. Define the nodal basis functions with respect to the partition $\Delta$ as
\begin{eqnarray*}
\phi_{i}(x;\bb)&=&\left\{\begin{array}{ll}
\dfrac{x - b_{i-1}}{h_{i-1}}, & x\in (b_{i-1},b_i), \\[4mm]
\dfrac{b_{i+1} - x}{h_{i}}, & x\in (b_i, b_{i+1}),\\[4mm]
0, & \text{otherwise},
\end{array}\right.
\end{eqnarray*}
for $i=1,...,n$ and
\begin{eqnarray*}
\phi_0(x;\bb)=\left\{\begin{array}{ll}
\dfrac{b_1 - x}{h_0}, & x\in (b_0, b_1),\\[4mm]
0, & \text{otherwise},
\end{array} \right.
\hspace{.5cm}
\phi_{n+1}(x;\bb)&=& \left\{\begin{array}{ll}
\dfrac{x - b_{n}}{h_{n}}, & x\in (b_{n}, b_{n+1}), \\[4mm]
0, & \text{otherwise}.
\end{array}\right.
\end{eqnarray*}
For each neural network function $v(x;\bb)\in\cM_n$, let $v_i = v(b_i;\bb)$ be the nodal value of $v(x;\bb)$ at $b_i$ for $i=0,1,...,n+1$. Then, $v(x;\bb)$ can be represented with the nodal basis functions as
\begin{equation}
    v(x;\bm{b}) = \sum\limits_{i=0}^{n+1} v_i \phi_i(x;\bb).
\end{equation}
Hence, the set $\cM_n$ defined in (\ref{func1}) is the same as the set of linear splines with $n$ free knots, which was studied intensively in the 1970s and 1980s (see, e.g., \cite{Schumaker}),
\begin{equation}\label{M_n}
    \cM_n =\left\{\sum\limits_{i=0}^{n+1} c_i \phi_i(x;\bb):\, c_i\in \R\,\mbox{ and }\, \bb=(b_1,\ldots,b_n)\in \R^n\right\}.
\end{equation}


\section{NN Representations of Initial and Boundary Data}\label{init_bdry_data}

First, the initial data $u_0(x)$ and the boundary data $g(t)$ are represented by continuous piecewise linear approximations. These approximations are computed by finding the best least squares approximations of $u_0(x)$ and $g(t)$ from the set of shallow ReLU neural network functions. Specifically, let 
\begin{eqnarray*}
    \cM_n^{(0)}(\Omega)&=& \big\{v\in \cM_n:\, v(a)=u_0(a) \,\mbox{ and }\, v(b)=u_0(b)\big\},\\[5mm]
\cM_n^{(t)}(I)&=&\{v\in\cM_n:\, v(0)=g(0)\,\mbox{ and }\, v(T)=g(T)\}.
\end{eqnarray*}
Then, find $u^{(0)}_N\in \cM_n^{(0)}$ and $g_N\in \cM_n^T$ such that 
\begin{equation}\label{init_ls_form}
 \left\| u^{(0)}_N - u_0\right\|_{L^2(\Omega)}= \min_{v \in \cM_n^{(0)}} \left\| v - u_0\right\|_{L^2(\Omega)}
 \,\mbox{ and }\,\left\| g_N - g\right\|_{L^2(I)}= \min_{v \in \cM_n^{(t)}} \left\| v - g\right\|_{L^2(I)},
\end{equation}
respectively.

The problems in (\ref{init_ls_form}) are solved using the adaptive network enhancement (ANE) method introduced in \cite{LiuCai1, LiuCai2}. With the ANE method, the initial data and the boundary data can be approximated within a target error tolerance $\varepsilon>0$ using relatively small neural networks. Hence, for a given $\varepsilon$, the $L^p$ norms of the numerical errors can be bounded by
\begin{equation}\label{error-initial}
    \left\|u_0- u^{(0)}_{N} \right\|_{L^p(\Omega)} \leq \varepsilon\, \left\|u_0\right\|_{L^p(\Omega)} 
    \quad\mbox{and}\quad \left\|g- g_N \right\|_{L^p(I)} \leq \varepsilon\, \left\|g \right\|_{L^p(I)}
\end{equation}
for $p$ in $[1,\infty)$, where $u^{(0)}_{N}(x)$ and $g_N(t)$ are the respective neural network approximations to $u_0(x)$ and $g(t)$. Let $\bb^{(0)}=\big(b^{(0)}_1,\ldots,b^{(0)}_{n_0}\big)$ and $\hat{\bt}=\big(\hat{t}_1,\ldots,\hat{t}_m\big)$ be the breaking points of $u^{(0)}_{N}(x)$ and $g_N(t)$ satisfying
\begin{equation}\label{u0g-pts}
a<b^{(0)}_1< \cdots <b^{(0)}_{n_0}< b \quad\mbox{and}\quad
0< \hat{t}_1 < \cdots < \hat{t}_m < T,
\end{equation}
respectively. Then, $u^{(0)}_{N}(x)$ and $g_N(t)$ have the following forms
\begin{equation}\label{initial-rep}
    u^{(0)}_{N}(x) = \sum\limits_{i=0}^{n_0+1} u^{(0)}_i \phi_i(t;\bb^{(0)})
    \quad\mbox{and}\quad 
    g_N(t) = \sum\limits_{i=0}^{m+1} g_i \phi_i(t;\hat{\bt}), 
\end{equation}
respectively, where $u^{(0)}_i=u_0(b^{(0)}_i)$ and $g_i=g(\hat{t}_i)$ are the nodal values of the respective initial and boundary data. 


\section{ENN Method for Linear Problems}\label{ENN-linear} This section describes the evolving neural network (ENN) method for the linear transport equation and it presents an error estimate based on the representation errors of the initial data and the boundary data.

To emulate the physical phenomena of the underlying problem, the ENN method propagates the NN representation of the initial profile along characteristic lines. This is equivalent to propagating the breaking points and the nodal values of the initial profile along characteristic lines. Let $\hat{t}$ denote the previous time of the approximation. The process of propagating a point $(\hat{b}, \bar{t} \,)$ in $\Omega \times [\hat{t}, t)$ with the nodal value $\hat{v}$ to the next time $t$ is described in Algorithm~\ref{char_prop_alg1}. 


\begin{algorithm}[H] 
\caption{Characteristic Scheme} \label{char_prop_alg1}
\begin{minipage}{.9\textwidth}
\vspace{.2cm}

Let $\hat{t}$ be the previous time, let $t$ be the next time, and let $f$ be the spatial flux.
Given 
a triple $(\hat{b}, \bar{t}, \hat{v})$ with $\hat{t}\leq \bar{t}$, compute a pair $(b,v)$ at the next time $t$ as follows
\[
b=\hat{b} + (t-\bar{t})f^\prime(\hat{v})\quad\mbox{and}\quad v=\hat{v}.
\]
\end{minipage}
\end{algorithm}

\begin{algorithm}[ht!]
\caption{ENN Method for Linear Problems} \label{enn_linear}
\begin{minipage}{.9\textwidth}
\vspace{.2cm}

Let $u^{(k-1)}_N(x)$ be the NN approximation at the previous time $t_{k-1}$ defined in (\ref{u(k-1)}).
Let $\big\{(\hat{b}_j,\hat{t}_j)\big\}_{\hat{t}_j\in (t_{k-1},t_k] }$ be the subset of the breaking points from the inflow boundary data. 
\begin{enumerate}
\item For each triple in $\left\{\!\left(b^{(k-1)}_i, t_{k-1}, u^{(k-1)}_i\right)\!\right\}_{i=0}^{n_{k-1}+1}\!\bigcup  \left\{\!\left(\hat{b}_j,\hat{t}_j,g\big(\hat{b}_j,\hat{t}_j)\right)\!\right\}_{\hat{t}_j\in [t_{k-1},t_k)}$, use Algorithm \ref{char_prop_alg1} to compute 
\[
\left\{\left(\tilde{b}^{(k)}_i, t_{k}, \tilde{u}^{(k)}_i\right)\right\}_{i=0}^{n_{k-1}+1}\bigcup  \left\{\left(\hat{b}^{(k)}_j,\hat{t}_j,\hat{u}^{(k)}_j\right)\right\}_{\hat{t}_j\in [t_{k-1},t_k)}.
\]

\item Combine and reorder the breaking points 
\[\left\{b^{(k)}_i\right\}_{i=1}^{n_{k}} \subseteq
   \left\{\tilde{b}^{(k)}_i\right\}_{i=1}^{n_{k-1}} \bigcup  \left\{\hat{b}^{(k)}_j\right\}_{\hat{t}_j\in [t_{k-1},t_k)}\]
    so that 
    \[
    b^{(k)}_0\leq a<b^{(k)}_1< \cdots < b^{(k)}_{n_{k}}<b \leq b^{(k)}_{n_k+1}. \]

\item For $i=1,\ldots,n_k$, set each $u^{(k)}_i$ computed in step 1 as the nodal value corresponding to $b^{(k)}_i$. For the nodal values on the boundaries, set
\begin{eqnarray*}
    u^{(k)}_0&=& \left\{\begin{array}{ll}
g(t_k), & b^{(k)}_0=a,\\[4mm]
u^{(k)}_0\dfrac{b^{(k)}_1-a}{b^{(k)}_1 - b^{(k)}_0} + u^{(k)}_1\dfrac{a-b^{(k)}_0}{b^{(k)}_1 - b^{(k)}_0}, & b^{(k)}_0<a
\end{array} \right.\\[2mm]
\mbox{and}\quad u^{(k)}_{n_k+1}&=& \left\{\begin{array}{ll}
g(t_k), & b^{(k)}_{n_k+1}=b,\\[4mm]
 u^{(k)}_{n_k}\dfrac{b^{(k)}_{n_k+1}-b}{b^{(k)}_{n_k+1} - b^{(k)}_{n_k}} + u^{(k)}_{n_k+1}\dfrac{b-b^{(k)}_{n_k+}}{b^{(k)}_{n_k+1} - b^{(k)}_{n_k}}, & b^{(k)}_{n_k+1}>b.
\end{array} \right.
\end{eqnarray*}

\item Output $u^{(k)}_N(x)$ defined in (\ref{u(k)}) as the NN approximation at time $t_k$.
\end{enumerate}
\end{minipage}
\end{algorithm}

\begin{remark}
Let $b_c$ be an end point of $\Omega=(a,b)$ such that $(b_c,0)$ is on the inflow boundary $\Gamma_-$. In the case that $u_0(b_c)\not= g(b_c)$, for $k=1$, {\em Algorithm~\ref{enn_linear}} needs a small modification. 
Specifically, shift the boundary point from $\hat{t}_0=0$ to $\hat{t}_0=\delta$ and the end point from $b_c=a$ or $b$ to $b_c=a+\delta$ or $b-\delta$ for a very small positive $\delta$.
\end{remark}

Next, we will describe how the ENN method uses characteristic propagation to compute the solution for the next time $t_k$ using the solution at the previous time $t_{k-1}$. 
To this end, let
\begin{equation}\label{u(k-1)}
    u^{(k-1)}_N(x)=\sum\limits_{i=0}^{n_{k-1}+1} u^{(k-1)}_i \phi_i(x;\bb^{(k-1)})
\end{equation}
be the NN approximation at time $t_{k-1}$ with the breaking points $\bb^{(k-1)}=\big(b^{(k-1)}_1,\ldots,b^{(k-1)}_{n_{k-1}}\big)$ satisfying
\[
a=b^{(k-1)}_0< b^{(k-1)}_1<\cdots<b^{(k-1)}_{n_{k-1}} <b^{(k-1)}_{n_{k-1}+1}=b.
\]
Let $\left\{\hat{t}_j\right\}_{j=1}^{m}$ be the breaking points of the inflow boundary data given in (\ref{u0g-pts}) with the corresponding $x$-coordinates $\big\{\hat{b}_j\big\}_{j=1}^{m}$. Then, the ENN method given in Algorithm~\ref{enn_linear} produces the NN approximation at the next time $t_k$ as 
\begin{equation}\label{u(k)}
    u^{(k)}_N(x)=\sum\limits_{i=0}^{n_k+1} u^{(k)}_i \phi_i(x;\bb^{(k)}),
\end{equation}
where $\bb^{(k)}$ and $\{u^{(k)}_i\}_{i=0}^{n_k+1}$ are the breaking points and the nodal values, respectively.

At time $t_{k}$, there are $n_k+2$ end points and breaking points. Suppose that it takes $K$ steps to propagate from time $t_0=0$ to time $t_K=T$. Then, the total number of end points and breaking points is about 
\[
N=2+m+2K+\sum\limits_{k=0}^{K-1} n_k.
\]
Since the cost of the propagating each point is three operations, the computational cost is about $3N$.

\begin{theorem}
Let $u(x,t)$ be the solution to {\em (\ref{pde})}, and let $u^{(k)}_N\left(x\right)$ be the NN approximation at time $t_k$ defined in {\em Algorithm~5.2}. For any $k=1,2,\ldots$,  
\begin{equation}\label{error-linear}
 \left\| u(\cdot,t_k)- u^{(k)}_N\right\|_{L^p(\Omega)} \leq \varepsilon\,\left(\left\|u_0\right\|_{L^p(\Omega)}^p+\left\|g\right\|_{L^p(I)}^p\right)^{1/p}.
 \end{equation}
\end{theorem}

\begin{proof}
At time $t_k$, partition the domain $\Omega$ into two disjoint sets $\Omega_{k,0}$ and $\Omega_{k,g}$ through the characteristic lines $x(t)=\alpha (t-\hat{t}) + x(\hat{t})$. Specifically, points in $\Omega_{k,0}\times \{t_k\}$ or $\Omega_{k,g}\times \{t_k\}$ are on the characteristic lines intersecting $\Omega\times \{0\}$ or $\{b_c\}\times I$, respectively. Because the solution $u(x,t)$ is a constant along each characteristic line, the construction of Algorithm \ref{enn_linear} implies that
\begin{equation}\label{4.10}
u\big(x(t_k),t_k\big)- u^{(k)}_N\big(x(t_k)\big)
= \left\{\begin{array}{ll}
     u_0\big(x(0)\big)- u^{(0)}_N\big(x(0)\big), & x(t_k)\in \Omega_{k,0}, \\[2mm]
    g\big(\hat{t}\big) - g_N\big(\hat{t}\big), & x(t_k)\in \Omega_{k,g},
    \end{array}
    \right.
\end{equation}
where $x(0)\in \Omega$ and $\hat{t} \in I$.
Using (\ref{4.10}) and (\ref{error-initial}), 
\begin{eqnarray*}
  \left\| u(\cdot,t_k)- u^{(k)}_N\right\|_{L^p(\Omega)}^p &=& \left\| u(\cdot,t_k)- u^{(k)}_N\right\|_{L^p(\Omega_{k,0})}^p + \left\| u(\cdot,t_k)- u^{(k)}_N\right\|_{L^p(\Omega_{k,g})}^p\\[2mm]
  &\leq & \left\| u_0- u^{(0)}_N\right\|_{L^p(\Omega)}^p + \| g-g_N\|_{L^p(I)}^p
  \leq \varepsilon^p \left(\left\|u_0\right\|_{L^p(\Omega)}^p+\left\|g\right\|_{L^p(I)}^p\right), 
\end{eqnarray*}
which implies the validity of (\ref{error-linear}). This completes proof of the theorem. 
\end{proof}

\section{ENN Method for Burgers' Equation}\label{ENN-nonlinear}

To extend the ENN method to Burgers' equation, a shock formation must be accurately simulated.
Since the partial differential equation in (\ref{pde}) is invalid at a shock, the characteristic propagation in Algorithm~\ref{char_prop_alg1} is also invalid. 
To circumvent this difficulty, we first identify {\it shock regions} which are relatively small regions containing the shock interfaces. Then, the solution is computed using characteristic propagation combined with the primitive form of the underlying problem in the shock regions. For any subdomain $V\subset \Omega\times I_k= \Omega\times (t_{k-1},t_k)$, the primitive form is given by
\begin{equation}\label{primitive}
      \int_{\partial V} \left( u n_t + \frac12 u^2 n_x\right)\,ds=0,
\end{equation}
where $(n_x,n_t)^T$ is the unit outward vector normal to the boundary of the volume $V$. 

\subsection{Shock Location and Shock Region}\label{s1} 

Let $u^{(k-1)}_N(x)$ be the approximation at time $t_{k-1}$ given as (\ref{u(k-1)}). For a given time $t_k>t_{k-1}$, a shock forms if the characteristics emanating from adjacent breaking points intersect in the time interval $(t_{k-1}, t_k)$. To simulate the shock formation, we define a shock region containing the shock interface in this section.

To this end, a spatial interval between two adjacent breaking points, $\left(b^{(k-1)}_l,b^{(k-1)}_{l+1}\right)$, is said to have a shock in the time interval $(t_{k-1}, t_k)$ if (i) the intersection of the characteristics emanating from $b^{(k-1)}_l$ and $b^{(k-1)}_{l+1}$ is the first intersection point of the characteristics from all breaking points and (ii) the length of the interval is less than or equal to a prescribed maximal distance $d^*$:
\begin{equation}\label{d_*}
   d_l^{(k-1)}:= b_{l+1}^{(k-1)} - b_l^{(k-1)} \leq d^{*} .
\end{equation}
Then, the location of the approximate shock at time $t_{k-1}$ is defined as the midpoint
\[
s_l^{(k-1)}=\dfrac12 \,\left(b^{(k-1)}_{l}+b^{(k-1)}_{l+1}\right).
\]

Given that the approximate shock location is in the interval $\left(b^{(k-1)}_l,b^{(k-1)}_{l+1}\right)$, we will describe the construction of the shock region. 
Let
\begin{equation}\label{b-tilde}
    \tilde{b}^{(k)}_{l+1}= b^{(k-1)}_{l} + \tau_k u^{(k-1)}_{l}
\quad\mbox{and}\quad 
\tilde{b}^{(k)}_{l}=b^{(k-1)}_{l+1} +\tau_k u^{(k-1)}_{l+1}
\end{equation}
denote the spatial coordinates of the characteristics emanating from $b^{(k-1)}_l$ and $b^{(k-1)}_{l+1}$ at time $t_k$, respectively. 
Then, the shock region, denoted by $V^{(k)}_l$ (see Figure \ref{shock-reg}), is constructed as the trapezoidal region with the following four vertices: 
\[
\big(b^{(k-1)}_l, t_{k-1}\big), \hspace{.3cm} \big(b^{(k-1)}_{l+1}, t_{k-1}\big), \hspace{.3cm} \big(\tilde{b}^{(k)}_{l}, t_k\big), \quad\mbox{and}\quad \big(\tilde{b}^{(k)}_{l+1}, t_k\big).
\]

\begin{figure}[H]
\centering
\caption{Shock Region}\label{shock-reg}
\begin{tikzpicture}[scale=1.2]
\draw[gray, thick] (-5,0)--(5,0);
\draw[gray, thick] (-5,2.5)--(5,2.5);

\draw[dashed, gray] (-.3,0)--(1.2,2.5);
\draw[dashed, gray] (.3,0)--(-1.2,2.5);


\draw[black, thick] (-.3,0)--(-1.2,2.5);
\draw[black, thick] (.3,0)--(1.2,2.5);

\draw[black, thick] (-.3,0)--(.3,0);
\draw[black, thick] (-1.2,2.5)--(1.2,2.5);

\filldraw[black] (-.3,0) circle (1.5pt) node[anchor=north]{$b^{(k-1)}_{l}$};
\filldraw[black] (.3,0) circle (1.5pt) ;

\draw[black] (.6,0) node[anchor=north] {$b_{l+1}^{(k-1)}$}; 

\filldraw[black] (-4,0) circle (1.5pt) node[anchor=north]{$b^{(k-1)}_{l-1}$};
\filldraw[black] (4,0) circle (1.5pt) node[anchor=north]{$b^{(k-1)}_{l+2}$};


\filldraw[black] (-2.5,2.5) circle (1.5pt) node[anchor=south]{${b}^{(k)}_{l-1}$};
\filldraw[black] (2,2.5) circle (1.5pt) node[anchor=south]{\qquad\,\,${b}^{(k)}_{l+2}$};

\draw[ultra thin, gray] (5,0) node[anchor=west] {$t_{k-1}$}; 
\draw[ultra thin, gray] (5,2.5) node[anchor=west] {$t_k$};

\draw[ultra thin, black] (0,1.5) node[anchor=center] {$V_l^{(k)}$};

\draw[dashed, gray] (-1.2, 2.5)--(-3,0);
\draw[dashed, gray] (1.2, 2.5)--(2.5,0);
\draw[black, fill=white] (-3,0) circle (1.5pt) node[anchor=north]{$\tilde{b}_{l}^{(k-1)}$};
\draw[black, fill=white] (-1.2,2.5) circle (1.5pt) node[anchor=south]{$\tilde{b}_{l}^{(k)}$};
\filldraw[black, fill=white] (1.2,2.5) circle (1.5pt) node[anchor=south]{$\tilde{b}^{(k)}_{l+1}$};
\filldraw[black, fill=white] (2.5,0) circle (1.5pt) node[anchor=north]{$\tilde{b}^{(k-1)}_{l+1}$};

\end{tikzpicture}
\end{figure}

\subsection{Propagation in Shock Region}\label{s2a}

In the shock region $V^{(k)}_l$ described in Section~\ref{s1}, the evolution of the solution is governed by the primitive form in \cref{primitive} with $V=V^{(k)}_l$. To accurately simulate shock propagation, we introduce an accurate scheme using both characteristics and the primitive form. 

To this end, consider the breaking points $b_{l-1}^{(k)}$ and $b_{l+2}^{(k)}$ to the left and right of the shock at time $t_k$ given by
\[ 
    b^{(k)}_{l-1}= 
    b^{(k-1)}_{l-1} + \tau_k u^{(k-1)}_{l-1}
\quad\mbox{and}\quad 
b^{(k)}_{l+2} = b^{(k-1)}_{l+2} +\tau_k u^{(k-1)}_{l+2},
\] 
where $\tau_k=t_k-t_{k-1}$ is the time step size. Assume that 
\begin{equation}\label{assumption1}
    b^{(k)}_{l-1} 
    <\tilde{b}^{(k)}_{l} <
\tilde{b}^{(k)}_{l+1}< 
b^{(k)}_{l+2}
\end{equation}
(see Fig \ref{shock-reg}). This condition depends on the time step size $\tau_k$ and will be discussed further in Section~\ref{sec_tsc}. Furthermore, assume that there is no shock in neighboring regions $V^{(k)}_{l-1}$ and $V^{(k)}_{l+1}$ (similarly defined as $V^{(k)}_{l}$).

Let $u^{(k)}_{N}(x)$ be the NN approximation at time $t_k$. Its restriction on $\left(b^{(k)}_{l-1}, b^{(k)}_{l+2}\right)$ is a continuous piecewise linear function with the nodal values $u^{(k)}_i=u^{(k)}_N\left(b^{(k)}_i\right)$ for $i=l-1,\ldots,l+2$. By the assumptions, it is natural to set 
\begin{equation}\label{bp-nv-V_l}
      u^{(k)}_{l-1}=u^{(k-1)}_{l-1}, \hspace{.3cm} u^{(k)}_{l+2}=u^{(k-1)}_{l+2},
    \quad\mbox{and}\quad  b^{(k)}_{l+1}=b^{(k)}_{l}+d_l 
\end{equation}
with $d_l=b^{(k-1)}_{l+1}-b^{(k-1)}_{l}$.
Below, we compute the nodal values $u^{(k)}_{l}$ and $u^{(k)}_{l+1}$ and the breaking point $b^{(k)}_{l}$.

Let
\[
m_l^{(k-1)}=\dfrac{{u}^{(k-1)}_{l}-{u}^{(k-1)}_{l-1}}{b_l^{(k-1)} - b_{l-1}^{(k-1)}}\quad\mbox{and}\quad 
m_{l+2}^{(k-1)} = \dfrac{{u}^{(k-1)}_{l+2}-{u}^{(k-1)}_{l+1}}{b_{l+2}^{(k-1)} - b_{l+1}^{(k-1)}}
\]
denote the derivatives of the NN approximation ${u}_N^{(k-1)}(x)$ at time $t_{k-1}$ in the intervals\\ $\left({b}^{(k-1)}_{l-1},\,{b}^{(k-1)}_{l}\right)$ and $\left({b}^{(k-1)}_{l+1},\,{b}^{(k-1)}_{l+2}\right)$, respectively. Let 
\[
\tilde{b}^{(k-1)}_{l}\in \big({b}^{(k-1)}_{l-1},\,{b}^{(k-1)}_{l}\big) \quad\mbox{and}\quad \tilde{b}^{(k-1)}_{l+1} \in \big({b}^{(k-1)}_{l+1},\,{b}^{(k-1)}_{l+2}\big)
\]
be the intersection points of the line $t=t_{k-1}$ with the respective characteristic lines backtraced from $\tilde{b}^{(k)}_{l}$ and $\tilde{b}^{(k)}_{l+1}$. 
Let $\tilde{u}^{(i)}_{j}$ be the nodal values of the NN approximation $u_N^{(i)}(x)$ at $\tilde{b}^{(i)}_{j}$ for $i=k-1,k$ and $j=l,l+1$.
Then,
\begin{equation}\label{u-til}
\left\{\begin{array}{l}
    \tilde{u}^{(k)}_{l} = \tilde{u}^{(k-1)}_{l} = {u}^{(k-1)}_{l} -m^{(k-1)}_{l}\left(b^{(k-1)}_{l}-\tilde{b}^{(k-1)}_{l}\right) \quad\mbox{and }\, \\[2mm]
    \tilde{u}^{(k)}_{l+1} = \tilde{u}^{(k-1)}_{l+1} = {u}^{(k-1)}_{l+1} +m^{(k-1)}_{l+2}\left(\tilde{b}^{(k-1)}_{l+1}-{b}^{(k-1)}_{l+1}\right),
    \end{array}\right.
\end{equation}
which, together with \cref{c-line}, implies that
\begin{equation}\label{bl_til}
 \tilde{b}^{(k-1)}_{l}\!=\!
 \dfrac{\tilde{b}^{(k)}_{l}\!+\!\tau_k\left(m_l^{(k-1)}{b}^{(k-1)}_{l}\!-\!u_{l}^{(k-1)}\right)}{1+\tau_k\, m_l^{(k-1)}}
 \,\mbox{ and }\,
 \tilde{b}^{(k-1)}_{l+1}\!=\!
 \dfrac{\tilde{b}^{(k)}_{l+1}\!+\!\tau_k\left(m_{l+2}^{(k-1)}{b}^{(k-1)}_{l+1}\!-\!u_{l+1}^{(k-1)}\right)}{1+\tau_k\, m_{l+2}^{(k-1)}},
 \end{equation}
provided that $1+\tau_km_l^{(k-1)}\not= 0$ and $1+\tau_km_{l+2}^{(k-1)}\not= 0$. 

Let
\begin{equation}\label{mk}
    m_l^{(k)}= \dfrac{m_l^{(k-1)}}{1+\tau_km_l^{(k-1)}} \quad\mbox{and}\quad 
m_{l+2}^{(k)} 
=\dfrac{m_{l+2}^{(k-1)}}{1+\tau_km_{l+2}^{(k-1)}}.
\end{equation}
Then, we define the nodal values $u^{(k)}_{l}$ and $u^{(k)}_{l+1}$ by
\begin{equation}\label{u}
u_{l}^{(k)}=\tilde{u}^{(k-1)}_{l} + m_l^{(k)} \left({b}^{(k)}_{l}-\tilde{b}^{(k)}_{l}\right)
\quad\mbox{and}\quad
 u_{l+1}^{(k)}=\tilde{u}^{(k-1)}_{l+1}  - m_{l+2}^{(k)} \left(\tilde{b}^{(k)}_{l+1}-{b}^{(k)}_{l} -d_l\right)
\end{equation}
as functions of the unknown breaking point ${b}^{(k)}_{l}$.

Denote the average and difference of ${u}^{(k-1)}_{l}$ and ${u}^{(k-1)}_{l+1}$ by
\begin{equation}\label{ubar-J}
\bar{u}^{(k-1)}_{l}=\frac12 \left({u}^{(k-1)}_{l}+{u}^{(k-1)}_{l+1}\right) \quad \mbox{and}\quad J^{(k-1)}_{l}={u}^{(k-1)}_{l}-{u}^{(k-1)}_{l+1}
\end{equation}
respectively. 

Next, we derive a formula for computing the breaking point ${b}^{(k)}_{l}$ based on the primitive form in \cref{primitive} with $V=V^{(k)}_l$.
To approximate the primitive form, the exact solution $u$ is replaced by the NN approximations at times $t_{k-1}$ and $t_{k}$. Also, integration of the spatial flux along the vertical boundaries of $V_l^{(k)}$ is approximated by the trapezoidal rule. Then, the NN approximation $u_N^{(k)}(x)$ is constructed to satisfy the resulting approximation to the primitive form:
\begin{equation}\label{b}
F_{l}^{(k)}\left(b_{l}^{(k)},u_{l}^{(k)},u_{l+1}^{(k)}\right)=r_{l}^{(k)},
\end{equation}
where $F_{l}^{(k)}$ and $r_{l}^{(k)}$ are given by
\begin{eqnarray*}
 2F_{l}^{(k)}\!\!\!\!&=& \!\!\!\!\left(b_{l}^{(k)} - \tilde{b}_{l}^{(k)}\right)\,\left(\tilde{u}_{l}^{(k-1)} + u_{l}^{(k)}\right) +d_l\left({u}_{l}^{(k)} + u_{l+1}^{(k)}\right)
 +\left(\tilde{b}_{l+1}^{(k)} - {b}_l^{(k)}-d_l\right)\left({u}_{l+1}^{(k)} + \tilde{u}_{l+1}^{(k-1)}\right) \quad\mbox{and}
 \\[4mm]
2r_{l}^{(k)} \!\!\!\!&=& \!\!\!\! 2 d_{l} \bar{u}_{l}^{(k-1)} +\tau_kJ^{(k-1)}_{l}\bar{u}^{(k-1)}_{l}-\dfrac{\tau_k}{2} \left[\big(\tilde{u}^{(k-1)}_{l+1}\big)^2-\big(\tilde{u}^{(k-1)}_{l}\big)^2\right] \\ [2mm]
   && \quad + 
   \tau_k\left({u}^{(k-1)}_{l}\tilde{u}^{(k-1)}_{l+1}-{u}^{(k-1)}_{l+1}\tilde{u}^{(k-1)}_{l}\right) - d_{l}\left(2\bar{u}^{(k-1)}_{l}+\tilde{u}_{l}^{(k-1)}+\tilde{u}_{l+1}^{(k-1)}\right) \\ [2mm]
   \!\!\!\!&=& \!\!\!\!\! \tau_kJ^{(k-1)}_{l}\bar{u}^{(k-1)}_{l} \!\! - \!\dfrac{\tau_k}{2}\! \left[\!\big(\tilde{u}^{(k-1)}_{l+1}\big)^2 \!\!\! -\! \big(\tilde{u}^{(k-1)}_{l}\big)^2\right] \!\! + \! \tau_k\!\left(\!{u}^{(k-1)}_{l}\tilde{u}^{(k-1)}_{l+1}\!\! - \!{u}^{(k-1)}_{l+1}\tilde{u}^{(k-1)}_{l}\!\right) \!\! - \! d_{l}\!\left(\tilde{u}_{l}^{(k-1)}\!+\tilde{u}_{l+1}^{(k-1)}\right).  
\end{eqnarray*}
The derivation of $2r_{l}^{(k)}$ uses the identities
\begin{equation}\label{b-til-k}
\tilde{b}_{l+1}^{(k)} - {b}_{l+1}^{(k-1)}=\tau_k {u}^{(k-1)}_{l}-d_l \quad\mbox{and}\quad \tilde{b}_{l}^{(k)} - {b}_{l}^{(k-1)}=\tau_k {u}^{(k-1)}_{l+1}+d_l.
\end{equation}
Hence,
\[
\tilde{b}^{(k)}_{l+1}-\tilde{b}^{(k)}_{l}=\tau_kJ^{(k-1)}_l-d_l.
\]

Substituting ${u}^{(k)}_{l}$ and ${u}^{(k)}_{l+1}$ given in (\ref{u}) into (\ref{b}) yields the following quadratic equation 
\begin{equation}\label{b1}
    \alpha \left({b}^{(k)}_{l}-\tilde{b}^{(k)}_{l}\right)^2 +\beta \left({b}^{(k)}_{l}-\tilde{b}^{(k)}_{l}\right) + \zeta=0,
\end{equation}
where the coefficients are given by
\begin{equation}\label{abz}
\qquad\quad\left\{\!\!\begin{array}{l}
\alpha=m_l^{(k)}\!-m_{l+2}^{(k)}, \\[4mm]
\beta=2\left(\tilde{u}_{l}^{(k-1)}\!-\tilde{u}_{l+1}^{(k-1)}\right) +d_l \left(m_l^{(k)}\!+\!m_{l+2}^{(k)}\right)\! + 2 m_{l+2}^{(k)}\left(\tau_kJ^{(k-1)}_l-2d_l\right), \text{ and } \\ [4mm]
\zeta=d_l\! \left(\tilde{u}_{l}^{(k-1)}\!+\!\tilde{u}_{l+1}^{(k-1)}\right) \!\!+\!\! \left(\!\tau_kJ^{(k-1)}_l\!-\!2d_l\right) \!\left(2\tilde{u}^{(k-1)}_{l+1}\!-\!m^{(k)}_{l+2}\big(\tau_kJ^{(k-1)}_l\!-\!d_l\big)\!\right)\!-\!2r_{l}^{(k)}.
\end{array}\right.
\end{equation}
Then, the breaking point $b_l^{(k)}$ is computed as the following solution to (\ref{b1}),
\begin{equation}\label{b2}
{b}^{(k)}_{l}=\left\{\begin{array}{ll}
          \tilde{b}^{(k)}_{l} -\dfrac{\zeta}{\beta}, & \text{if }\, \alpha=0,\\ [4mm]
          \tilde{b}^{(k)}_{l} +\dfrac{1}{2\alpha}\left(-\beta {\pm}\sqrt{\beta^2-4\alpha \zeta}\right),  & \text{if }\, \alpha\not=0,
         \end{array}\right. 
\end{equation}
where the sign $+$ or $-$ is chosen so that ${b}^{(k)}_{l}\in \left[\tilde{b}^{(k)}_{l},\tilde{b}^{(k)}_{l+1}\right]$. 


\begin{algorithm}[H]
\caption{Characteristic Finite Volume Scheme} \label{char_fv}
\begin{minipage}{.9\textwidth}
\vspace{.2cm}

Let $u^{(k-1)}_N(x)$ be the NN approximation at the previous time $t_{k-1}$ defined in (\ref{u(k-1)}). Let $\left(b^{(k-1)}_l, b^{(k-1)}_{l+1}\right)$ be the interval at time $t_{k-1}$ with a shock. Assume that condition (\ref{assumption1}) 
holds, let $d_l = b_{l+1}^{(k-1)} - b_l^{(k-1)}$, and proceed with following steps:

\begin{enumerate}
\item Using (\ref{b-tilde}), compute 
\[
\tilde{b}^{(k-1)}_{l}\in \big({b}^{(k-1)}_{l-1},\,{b}^{(k-1)}_{l}\big) \quad\mbox{and}\quad\tilde{b}^{(k-1)}_{l+1} \in \big({b}^{(k-1)}_{l+1},\,{b}^{(k-1)}_{l+2}\big).
\]
\item Using (\ref{b2}), compute the breaking point ${b}^{(k)}_{l}$.
\item Using (\ref{u}), compute the nodal values $u_{l}^{(k)}$ and $u_{l+1}^{(k)}$ at the breaking points ${b}^{(k)}_{l}$ and ${b}^{(k)}_{l+1}={b}^{(k)}_{l}+d_l$, respectively.
\end{enumerate}
\end{minipage}
\end{algorithm}

\subsection{Time Step Control}\label{sec_tsc} 

Conditions (\ref{d_*}) and (\ref{assumption1}) depend on the prescribed current time step size $\tau_k$. This section investigates a proper time step size when either condition  (\ref{d_*}) or (\ref{assumption1}) is invalid. Set $d^{(k-1)}_l=b_{l+1}^{(k-1)} - b_l^{(k-1)}$.

In the case that (\ref{d_*}) is invalid, a small time step is needed before the shock location can be defined (see Figure \ref{ts-fig}). Specifically, the prescribed time step size $\tau=\tau_k$ is truncated to \[
\hat{\tau}_k = (d_l^{(k-1)}-d^*)/J^{(k-1)}_l,
\]
where $J^{(k-1)}_l$ is defined in \cref{ubar-J}. This $\hat{\tau}_k$ gives 
\[
d_l^{(k)}:=b_{l+1}^{(k)} - b_l^{(k)}=\left(b_{l+1}^{(k-1)}+\hat{\tau}_k u_{l+1}^{(k-1)}\right)-\left(b_{l}^{(k-1)}+\hat{\tau}_k u_{l}^{(k-1)}\right)
=d^*,
\]
where $d^*$ is a prescribed maximal distance between breaking points that represent a shock.
The newly formed shock at time $t_k=t_{k-1}+\hat{\tau}_k$ is defined as the midpoint of interval $\left({b}^{(k)}_{l},{b}^{(k)}_{l+1}\right)$
\begin{equation}\label{s-location}
    s_l^{(k)} = \dfrac{1}{2}\left( b_l^{(k)} + b_{l+1}^{(k)}\right).
\end{equation}


\begin{figure}[H]\label{ts-fig}
\centering
\caption{Truncated Time Step Before Shock Formation}
\begin{tikzpicture}[scale=1.2]
\draw[gray, thick] (-5,0)--(5,0);
\draw[gray, thick] (-5,2.5)--(5,2.5);
\draw[gray, thick] (-5, .65)--(5,.65);

\draw[dashed, gray] (-.9,0)--(1.1,2.5);
\draw[dashed, gray] (.5,0)--(-.6,2.5);
\draw[dashed, gray] (-2,0)--(0,2.5);
\draw[dashed, gray] (3,0)--(1.8, 2.5);

\filldraw[black] (-.9,0) circle (1.5pt) node[anchor=north]{$b^{(\!k\!-\!1\!)}_{l}$};
\filldraw[black] (.5,0) circle (1.5pt) node[anchor=north]{\quad $b^{(\!k\!-\!1\!)}_{l+1}$};
\filldraw[black] (-2,0) circle (1.5pt) node[anchor=north]{$b^{(k-1)}_{l-1}$};
\filldraw[black] (3,0) circle (1.5pt) node[anchor=north]{$b^{(k-1)}_{l+2}$};

\filldraw[red] (-.36,.65) circle (1.5pt);
\filldraw[red] (.2,.65) circle (1.5pt);
\filldraw[black] (-1.47, .65) circle (1.5pt);
\filldraw[black] (2.7, .65) circle (1.5pt);

\draw[ultra thin, gray] (-.36, .47)--(.2, .47);
\draw[ultra thin, gray] (-.36, .45)--(-.36, .5);
\draw[ultra thin, gray] (.2, .45)--(.2, .5);
\draw[ultra thin, gray] (-.06, .47) node[anchor=north]{$d^*$};

\draw[decorate, gray, thick, decoration = {brace}] (-5.8,-.1) --  (-5.8,2.4);
\node[gray, rotate=90] at (-6.1, 1.25){$\tau_k$};

\draw[decorate, gray, thick, decoration = {brace}] (-5.2,-.1) --  (-5.2,.67);
\node[gray, rotate=90] at (-5.5, .34){$\hat{\tau}_k$};

\draw[ultra thin, gray] (5,0) node[anchor=west] {$t_{k-1}$}; 
\draw[ultra thin, gray] (5,.65) node[anchor=west] {$t_{k}$}; 

\end{tikzpicture}
\end{figure}

In the case that (\ref{assumption1}) is invalid, 
a dramatic change in the solution occurs within the time step, and hence the time step size needs to be reduced. Below, we determine a proper time step size based on how quickly the breaking points merge into the shock.

To this end, let $x_{i}(t)$ denote the characteristic line emanating from the breaking point $b_{i}^{(k-1)}$ for $i=0,1, \ldots, n_{k-1}+1$. Assume that there is a shock in the interval $\left( b_l^{(k-1)}, b_{l+1}^{(k-1)}\right)$. 
Let $t_{k-1}+t^*$ denote the time that 
the characteristic lines ${x}_l(t)$ and ${x}_{l+1}(t)$ 
intersect. Then, $t^*$ is given by 
\begin{equation}\label{t*}
 t^* = \dfrac{b_{l+1}^{(k-1)} - b_l^{(k-1)}} { u_l^{(k-1)} - u_{l+1}^{(k-1)}}.
\end{equation}

Similarly, $t_{k-1}+t_{i,l+1}$ and $t_{k-1}+t_{l,j}$
are the times when the lines ${x}_{i}(t)$ and ${x}_{l+1}(t)$ and the lines ${x}_l(t)$ and ${x}_{j}(t)$ intersect, respectively, where $t_{i,l+1}$ and $t_{l,j}$ are given by
\[
 t_{i,l+1} = \dfrac{b_{{l+1}}^{(k-1)} - b_{i}^{(k-1)}} { u_{i}^{(k-1)} - u_{{l+1} }^{(k-1)}} 
 \quad\mbox{and}\quad  t_{l,j}= \dfrac{b_{l}^{(k-1)} - b_{j}^{(k-1)}} { u_{j}^{(k-1)} - u_{l}^{(k-1)}} 
\]
for $i=l-1,\ldots,1$ and $j=l+2,\ldots,n_{k-1}$.
After a shock formation, it is desirable to have a time step size $\tau_k$ such that
\[
x_l(t_{k-1}+\tau_k) - x_{l+1}(t_{k-1}+\tau_k)\geq d_l^{(k-1)}.
\]
Then, since $x_l(t_{k-1} + 2t^*) - x_{l+1}(t_{k-1}+2t^*)=d_l^{(k-1)}$, the minimum time step size is $2t^* $.


\begin{figure}[H]
\centering
\caption{Breaking Points Merging Into a Shock}\label{2_t_star}
\begin{tikzpicture}[scale=1.2]
\draw[gray, thick] (-5,0)--(5,0);
\draw[gray, thick] (-5,3)--(5,3);
\draw[gray] (-5,1.1)--(5,1.1); 

\draw[lightgray] (-.5,0)--(2.25,3);
\draw[lightgray] (.5,0)--(-2.25,3);

 
\draw[dashed, lightgray] (-.85,0)--(1,3);
\draw[dashed, lightgray] (-.65,0)--(1.6,3);

\draw[dashed, darkgray] (-.5,1.1)--(-1.15, 0);
\draw[dashed, darkgray] (.5,1.1)--(1.2,0);
\draw[dashed, darkgray] (-2.5,0)--(-1.7,3);
\draw[dashed, darkgray] (-4,0)--(-3.75,3);
\draw[dashed, darkgray] (4,0)--(3.15,3);

\filldraw[red] (-.85,0) circle (1.5pt);
\filldraw[red] (-.65,0) circle (1.5pt);

\filldraw[black, fill=white] (-.5,1.1) circle (1.5pt);
\filldraw[black, fill=white](.5,1.1) circle (1.5pt);
\filldraw[black, fill=white] (-1.15,0) circle (1.5pt);
\filldraw[black, fill=white] (1.2,0) circle (1.5pt);
\draw[black] (-.35,1.1) node[anchor=south] {$\tilde{b}_l^{(k)}$};
\draw[black] (.6,1.1) node[anchor=south] {$\tilde{b}_{l+1}^{(k)}$};
\draw[black] (-1.3,0) node[anchor=north] {$\tilde{b}_l^{(k-1)}$};
\draw[black] (1.6,0) node[anchor=north] {$\tilde{b}_{l+1}^{(k-1)}$};

\filldraw[black] (-2.2,1.1) circle (1.5pt);
\filldraw[black] (-3.9,1.1) circle (1.5pt);
\filldraw[black] (3.7, 1.1) circle (1.5pt);

\filldraw[black] (-.5,0) circle (1.5pt);
\filldraw[black] (.5,0) circle (1.5pt);
\filldraw[black] (-2.5,0) circle (1.5pt);

\draw[black] (-.35,0) node[anchor=north] {$b_l^{(k-1)}$};
\draw[black] (.6,0) node[anchor=north] {$b_{l+1}^{(k-1)}$}; 

\filldraw[black] (-4,0) circle (1.5pt);
\filldraw[black] (4,0) circle (1.5pt);

\draw[ultra thin, gray] (5,0) node[anchor=west] {$t_{k-1}$}; 
\draw[ultra thin, gray] (5,1.1) node[anchor=west] {$t_{k-1} + 2t^*$};

\end{tikzpicture}
\end{figure}

The time step control procedure is separated into two cases based on the times that the characteristics of breaking points intersect the lines $x_l$ and $x_{l+1}$. First, if $\min\{t_{l-1,l+1}, t_{l,l+2}\} >2t^*$, set the time step size as
\[
\tau_k= \dfrac12\,\left(\min\{t_{l-1,l+1}, t_{l,l+2}\}  +2t^*\right). 
\]
Then, the unknown parameters $\big(b_l^{(k)}, u_l^{(k)}, u_{l+1}^{(k)}\big)$ are computed using the characteristic finite volume scheme defined in Algorithm~\ref{char_fv}. 

Next, if $\min\{t_{l-1,l+1}, t_{l,l+2}\} \leq 2 t^*$ (see Figure \ref{2_t_star}), set 
\begin{equation}\label{2t^*}
\tau_k=2t^* \quad\mbox{and}\quad  b_l^{(k)}=\tilde{b}_l^{(k)}.
\end{equation}
To compute $u_l^{(k)}$, find the smallest $i\leq l$ such that $t_{i,l+1} \leq 2 t^*$, and let $\tilde{b}^{(k-1)}_{l}$ be the solution of the following characteristic line equation
\begin{equation}\label{char_1}
{b}^{(k)}_{l}=\tilde{b}^{(k-1)}_{l} +2t^* \,\left[{u}^{(k-1)}_{i-1} +m_i^{(k-1)}\,\big(\tilde{b}^{(k-1)}_{l}-{b}^{(k-1)}_{i-1}\big)\right],
\end{equation}
and then set 
\[
u_l^{(k)}=u_N^{(k-1)}\left(\tilde{b}^{(k-1)}_{l}\right).
\]
Note that the breaking points $b_i^{(k-1)}, b_{i+1}^{(k-1)},\ldots,b_{l}^{(k-1)}$ merge into one breaking point $b_l^{(k)}$ at time $t_k$. 

In a similar fashion, to compute $u_{l+1}^{(k)}$, find the largest $j\ge l+1$ such that $t_{l,j} \leq 2 t^*$, and
let $\tilde{b}^{(k-1)}_{l+1}$ be the solution of the following characteristic line equation
\begin{equation} \label{char_2}
{b}^{(k)}_{l+1}=\tilde{b}^{(k-1)}_{l+1} +2t^* \,\left[{u}^{(k-1)}_{j} +m_{j+1}^{(k-1)}\,\big(\tilde{b}^{(k-1)}_{l+1}-{b}^{(k-1)}_{j}\big)\right],
\end{equation}
and then set 
    \[
    u_{l+1}^{(k)}=u_N^{(k-1)}\left(\tilde{b}^{(k-1)}_{l+1}\right).
    \]
Again, the breaking points $b_{l+1}^{(k-1)}, b_{l+2}^{(k-1)},\ldots,b_{j}^{(k-1)}$ merge into one breaking point $b_{l+1}^{(k)}$ at time $t_k$.

\subsection{ENN Method for Burgers' Equation} 

With the characteristic finite volume scheme and the time step control procedure, we are now ready to describe the ENN method for Burgers' equation.  


\begin{algorithm}[H]
\caption{ENN Method for Burgers' equation} \label{enn_nonlinear}
\begin{minipage}{.9\textwidth}
\vspace{.2cm}

Let $u^{(k-1)}_N(x)$ be the NN approximation at the previous time $t_{k-1}$ defined in (\ref{u(k-1)}), let $t_k = t_{k-1} + \tau_k$ for a prescribed time step size $\tau_k$, and proceed with the following steps:\\
\begin{enumerate}

\item Using the conditions described in Section \ref{s1}, determine $S^{(k-1)}$, which is the set of pairs of breaking points containing shocks at time $t_{k-1}$. If a shock is detected in the time interval $(t_{k-1}, t_k)$, but condition \cref{d_*} is not satisfied, truncate the time step size for that shock to 
$
\tau_k = (d_l - d^*)/ J_l^{(k-1)}
$
as described in Section \ref{sec_tsc}.\\

\item For each pair $\left(b_l^{(k-1)}, b_{l+1}^{(k-1)}\right)$ in $S^{(k-1)}$, compute $\tilde{b}_l^{(k-1)}$ and $\tilde{b}_{l+1}^{(k-1)}$ using \cref{b-tilde}. If condition \cref{assumption1} holds, compute the unknown parameters $p^{(k)} = \left(b_l^{(k)}, u_l^{(k)}, u_{l+1}^{(k)}\right)$ with Algorithm \ref{char_fv}. Otherwise, as described in Section \ref{sec_tsc}, let $t_{\text{min}}=\min \left\{ t_{l-1, l+1}, t_{l,l+2} \right\}$ and compute $p^{(k)}$ as follows:\\
\begin{itemize}
\item[(i)] If $t_{\text{min}}>2t^*$, set $\tau_k = \frac{1}{2}\left( t_{\text{min}}  + 2t^*\right)$ and compute $p^{(k)}$ with Algorithm \ref{char_fv}.
\item[(ii)] if $t_{\text{min}} \leq 2t^*$, set $\tau_k = 2t^*$, and let
\[
p^{(k)} = \left( \tilde{b}_l^{(k)}, \, u_N^{(k-1)}\left(\tilde{b}_l^{(k-1)}\right), \, u_N^{(k-1)}\left(\tilde{b}_l^{(k-1)}\right)    \right),
\]
where $\tilde{b}_l^{(k-1)}$ and $\tilde{b}_{l+1}^{(k-1)}$ are the solutions to \cref{char_1} and \cref{char_2}, respectively.\\
\end{itemize}

\item For breaking points not in $S^{(k-1)}$ that do not merge into a shock, compute the corresponding breaking points and their nodal values using Algorithm \ref{char_prop_alg1}.\\

\item Output $u_N^{(k)}$ defined in (\ref{u(k)}) as the NN approximation at time $t_k$.
\end{enumerate}
\end{minipage}
\end{algorithm}


\section{Numerical Experiments}\label{num-exp} This section presents the numerical results of the ENN method for the one-dimensional linear advection equation and the inviscid Burgers equation. For all experiments, the network approximations to the initial and boundary data are constructed to a given error tolerance using the Adaptive Network Enhancement (ANE) Method \cite{LiuCai1}.

Let $u$ be the exact solution of problem (\ref{pde}), and let $u^{(k)}_N$ be the NN approximation at time $t_k$. Tables \ref{tab_two_disc}-\ref{tab_burg_bell} report the relative numerical errors of the NN approximation at several selected times in the $L^2$ norm. Figures \ref{fig_two_disc}-\ref{fig_burg_bell} show the NN approximation and the exact solution or reference solution at several selected times, where the breaking points of the approximation marked below.  


\subsection{The Linear Advection Equation}  This section reports the results of the ENN method for the one-dimensional linear advection equation with the following spatial flux 
\[
f(u)=u. 
\]
In this case, the initial profile propagates unchanged to the right at a constant velocity. 
The numerical error at each time step is bounded by the representation error of the initial data and the boundary data as shown theoretically in (\ref{error-linear}) and numerically in Tables \ref{tab_two_disc} and \ref{tab_pw_smooth}. These tables report the relative numerical error in the $L^2$ norm at several selected times. Moreover, the ENN method approximates the discontinuous solution well with no oscillations as shown in Figures \ref{fig_two_disc} and \ref{fig_pw_smooth}. These figures depict the NN approximation with its breaking points and the exact solution.

{
\renewcommand{\arraystretch}{1.1}
\begin{table}[!htb] \label{tab_two_disc}
\setlength{\tabcolsep}{10pt}
\caption{Relative errors of the problem with discontinuous and sinusoidal initial data}
\vspace{6pt}
\begin{center}
\begin{tabular}{|l|c|c|}
\hline
Time & $\frac{\|u(\cdot, t_k)- u^{(k)}_{N}\|_{L^2(\Omega)}}{\|u(\cdot, t_k)\|_{L^2(\Omega)}}$  \\ \hline 
0.00 & $2.5688\times 10^{-2}$\\ \hline
0.25 &  $2.5687\times 10^{-2}$ \\ \hline
0.50  &  $2.5687\times 10^{-2}$ \\ \hline
\end{tabular}
\end{center}
\end{table}
}


\begin{figure}[H]\label{fig_two_disc}
\centering
\subfigure[Initial data approximation]{
\begin{minipage}[t]{0.48\linewidth}
\includegraphics[scale=.47]{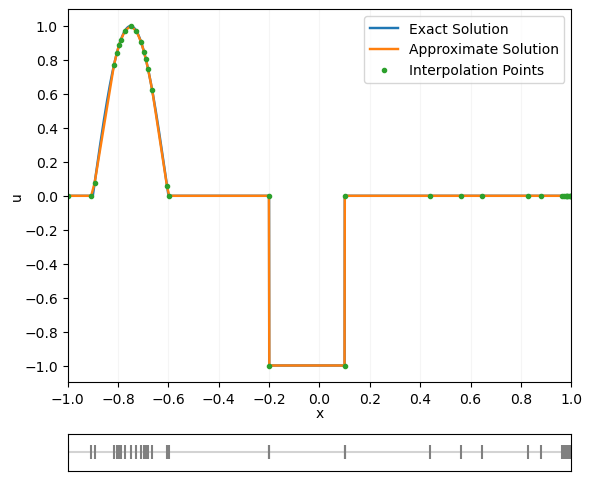}
\end{minipage}%
}%
\subfigure[Approximation at $t=0.5$]{
\begin{minipage}[t]{0.48\linewidth}
\centering
\includegraphics[scale=.47]{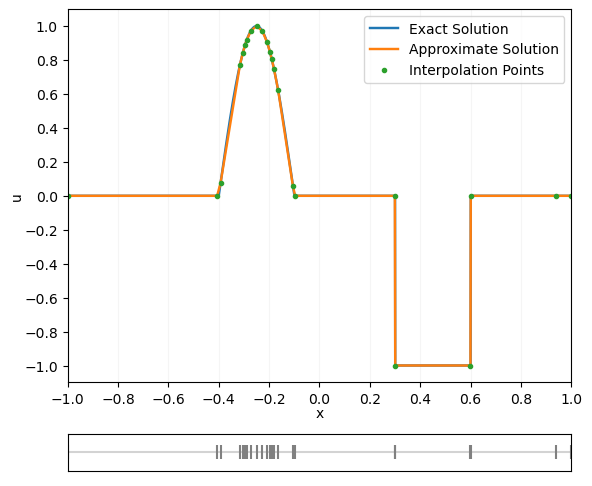}
\end{minipage}%
}%
\caption{Approximation results of the problem with discontinuous and sinusoidal initial data}
\end{figure}

\subsubsection{Discontinuous and Sinusoidal Initial Data} The second test problem has the domain $\Omega = (-1,1)$, the time interval $I = (0,0.5)$, and the inflow boundary $\Gamma_{-} = \left\{(-1, t): t \in I   \right\}$. The inflow boundary data and the initial data are given by
\[
g(x,t) = 0
\hspace{.25cm}
\text{and}
\hspace{.25cm}
u_0(x) = \left\{
\begin{array}{ll}
\dfrac{\sin(\pi (x + 0.9))}{0.3}, & -0.9 < x < -0.6,\\[3mm]
-1, & -0.2 < x < 0.1,\\[2mm]
0, & \text{ otherwise, }
\end{array}
\right.
\]
respectively. The exact solution is given by
\[
u(x,t) = \left\{
\begin{array}{ll}
\dfrac{\sin(\pi (x - t + 0.9))}{0.3}, & -0.9 < x -t < -0.6,\\[3mm]
-1, & -0.2 < x -t < 0.1,\\[2mm]
0, & \text{ otherwise }
\end{array}
\right. 
\]
for $(x,t)$ in $\Omega \times I$. The NN representation of the initial data was computed with the ANE method such that the relative numerical error in the $L^2$ norm was less than $\epsilon=3.0 \times 10^{-2}$. The resulting NN representation of the initial data had 37 neurons. 



{
\renewcommand{\arraystretch}{1.1}
\begin{table}[H]\label{tab_pw_smooth}
\setlength{\tabcolsep}{10pt}
\caption{Relative errors of the problem with a piecewise smooth solution}
\vspace{6pt}
\centering
\begin{tabular}{|c|c|c|}
\hline
Time & $\frac{\|u(\cdot, t_k)-u^{(k)}_{N}\|_{L^2(\Omega)}}{\|u(\cdot, t_k)\|_{L^2(\Omega)}}$ \\ \hline
0.00 & $8.2421\times 10^{-4}$ \\ \hline
0.25 & $6.86119\times 10^{-4}$  \\ \hline
0.50  & $8.8717\times 10^{-4}$  \\ \hline
0.75 & $6.0632\times 10^{-4}$  \\ \hline
1.00  & $5.2592\times 10^{-4}$ \\ \hline
\end{tabular}
\end{table}
}

\begin{figure}[H]\label{fig_space_time_pw_smooth}
\centering
\includegraphics[scale=.7]{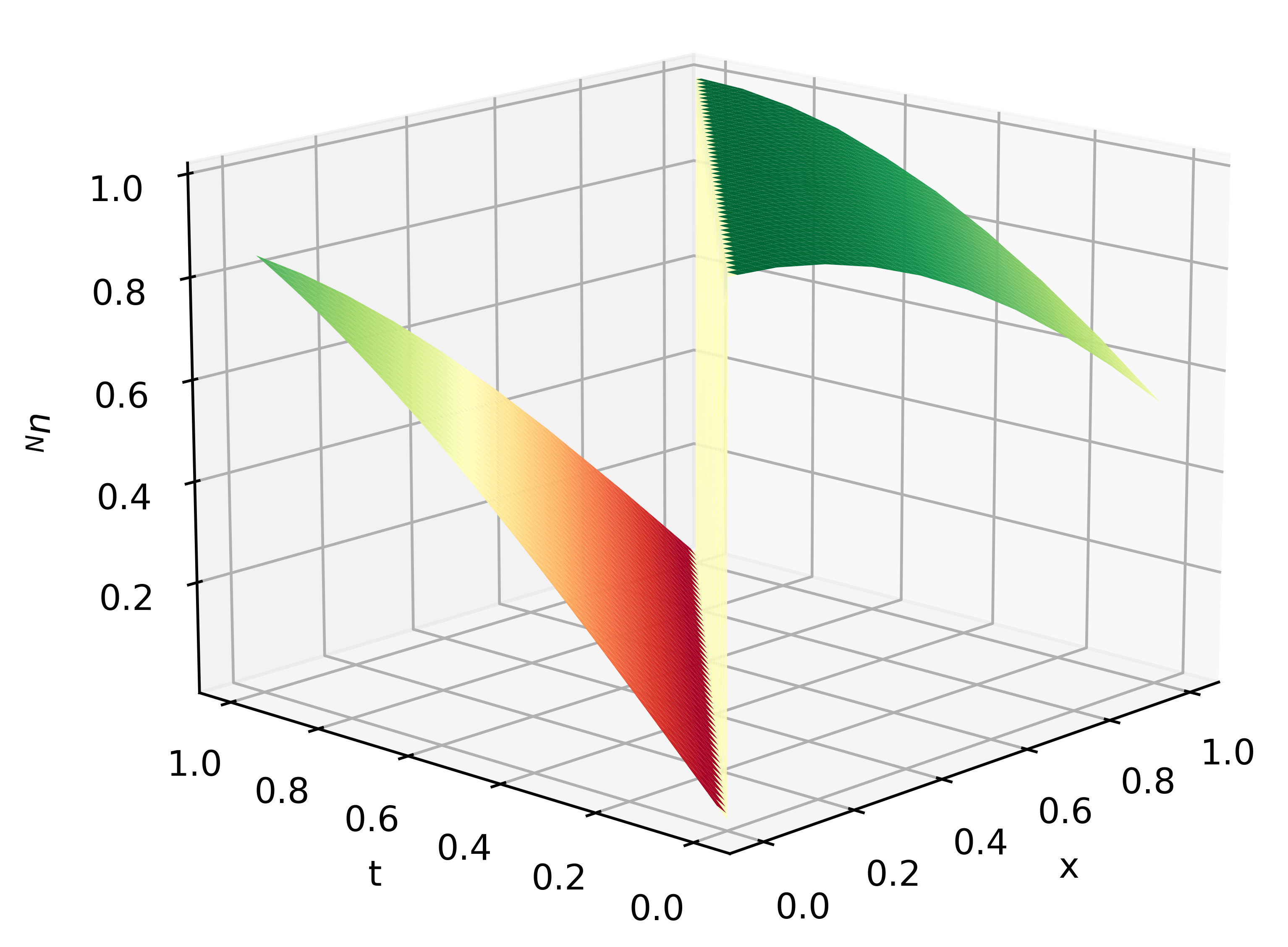}
\caption{Approximation results of the problem with a piecewise smooth solution in $\Omega \times I$}
\end{figure}

\subsubsection{Piecewise Smooth Solution} The third test problem has 
the domain $\Omega = (0,1)$, the time interval $I = (0,0.5)$, and the inflow boundary $\Gamma_{-} = \left\{(-1, t): t \in I \right\} $. The inflow boundary data is given by $g(x,t) = \sin (t)$, and the initial data is given by
$
u_0(x) = \cos(x).
$
The exact solution is given by
\[
u(x,t) = \left\{
\begin{array}{ll}
\sin(t-x) , & x<t, \\[2mm]
\cos (x-t), & x>t
\end{array}
\right.
\]
for $(x,t)$ in $\Omega \times I$.

The NN representations of the initial data and the boundary data were computed with the ANE method such that the relative numerical error in the $L^2$ norm was less than $\epsilon=3.0 \times 10^{-3}$. 
 
The relative numerical error of the approximation to the boundary data in the $L^2$ norm was $5.5285\times 10^{-4}$. The NN approximation to the initial data had 16 neurons and the NN approximation to the boundary data had 14 neurons. Figure \ref{fig_space_time_pw_smooth} shows the NN approximations in $\Omega \times I$. 

\begin{figure}[H] \label{fig_pw_smooth}
\centering

\subfigure[Boundary data approximation]{
\begin{minipage}[t]{0.47\linewidth}
\centering
\includegraphics[scale = .45]{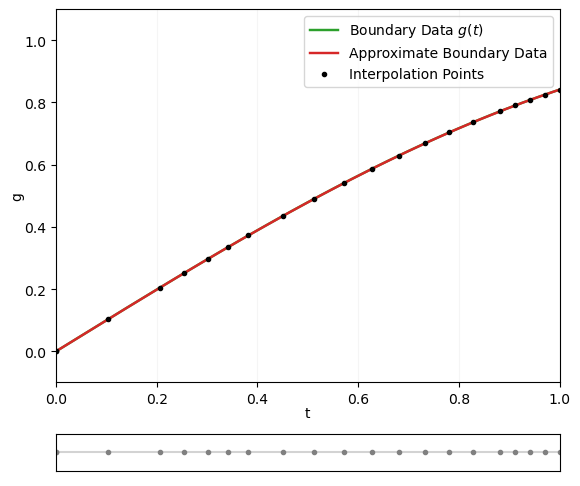}
\end{minipage}
}
\subfigure[Initial data approximation]{
\begin{minipage}[t]{0.47\linewidth}
\centering
\includegraphics[scale=.45]{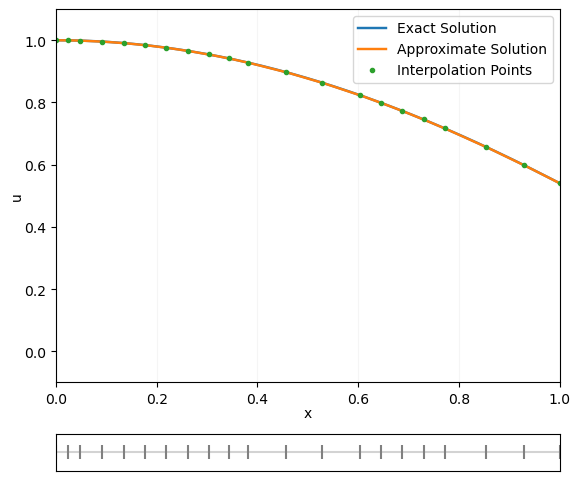}
\end{minipage}%
}%

\subfigure[Approximation at $t=0.25$]{
\begin{minipage}[t]{0.47\linewidth}
\centering
\includegraphics[scale = .45]{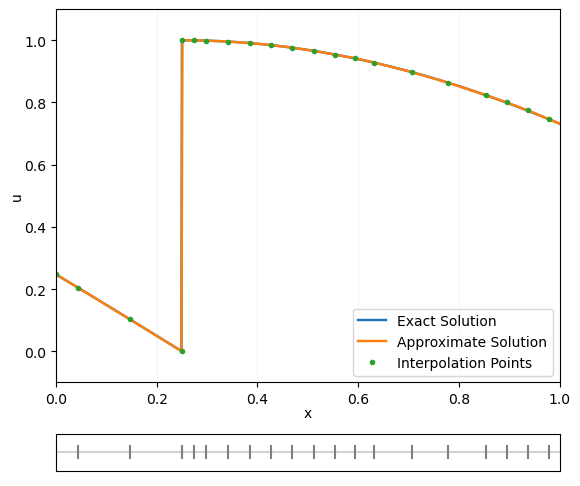}
\end{minipage}
}
\subfigure[Approximation at $t=0.75$]{
\begin{minipage}[t]{0.47\linewidth}
\centering
\includegraphics[scale=.45]{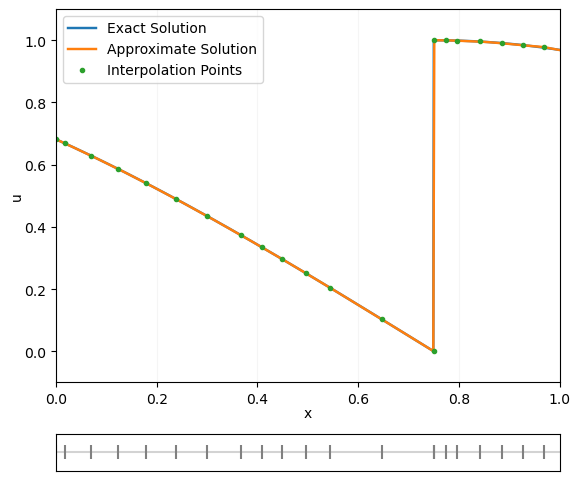}
\end{minipage}%
}%

\caption{Approximation results of the problem with a piecewise smooth solution}
\end{figure}


\subsection{Inviscid Burgers' Equation} This section reports the results of the ENN method for the one-dimensional inviscid Burgers equation, which has the following quadratic spatial flux 
\[
f(u) = \frac{1}{2}u^2.
\]
In this case, characteristic lines can intersect causing a shock formation. For these test problems, the ENN method employs the characteristic finite volume scheme to compute the shock location and jump. The time step control procedure described in Section \ref{sec_tsc} is used to capture dramatic changes in the solution and to merge any breaking points that intersect the shock interface. 

Because the solutions to the problems defined in Sections \ref{burg-sine} and \ref{burg-exp} are unknown, benchmark reference solutions are used to measure the quality of the NN approximations. Each benchmark reference solution, denoted by $\tilde{u}$, is computed using the third-order accurate WENO scheme and the fourth-order Runge-Kutta method on a fine mesh. The mesh size used to compute the reference solution was $\Delta x = 10^{-3}$ and $\Delta t = 2.0 \times 10^{-4}$ in the spatial and temporal directions, respectively. Because the prescribed 

Tables \ref{tab_rp_rarefaction}, \ref{tab_burg_sine}, and \ref{tab_burg_bell} report the relative numerical error in the $L^2$ norm at several selected times. Tables \ref{tab_burg_sine} and \ref{tab_burg_bell} also report the number of breaking points of the NN approximation at selected times for problems that implement time step control and merge breaking points that intersect the shock interface. Tables \ref{computation-1} and \ref{computation-2} report the number of mesh points and time steps required to compute the numerical approximation 
at the final times using the ENN method and the WENO scheme.

Moreover, the ENN method can accurately capture a shock's speed and height as shown in Figures \ref{fig_burg_sine} and \ref{fig_burg_bell}. Figure \ref{fig_rarefaction} demonstrates that the NN approximation computes the vanishing viscosity solution without special treatment. All figures in this section depict the NN approximation with its breaking points and the exact solution.

{
\renewcommand{\arraystretch}{1.1}
\begin{table}[!htb]\label{tab_rp_rarefaction}
\setlength{\tabcolsep}{10pt}
\caption{Relative errors of the Riemann problem with rarefaction for Burgers' equation}
\vspace{6pt}
\begin{center}
\begin{tabular}{|c|c|c|}
\hline
Time & $\frac{\|u(\cdot, t_k)- u^{(k)}_{N}\|_{L^2(\Omega)}}{\|u(\cdot, t_k)\|_{L^2(\Omega)}}$  \\ \hline 
0.0 & $1.9775 \times 10^{-2}$ \\ \hline
0.1 &  $1.6860\times 10^{-3}$\\ \hline
0.2  &  $1.2393\times 10^{-3}$ \\ \hline
0.3  & $1.0538\times 10^{-3}$ \\ \hline
0.4  & $9.5347\times 10^{-4}$ \\ \hline
0.5  & $8.9459\times 10^{-4}$ \\ \hline
\end{tabular}
\end{center}
\end{table}
}


\begin{figure}[!htb]\label{fig_rarefaction}
\centering
\subfigure[Approximation at $t=0.0$]{
\begin{minipage}[t]{0.48\linewidth}
\includegraphics[scale = .47]{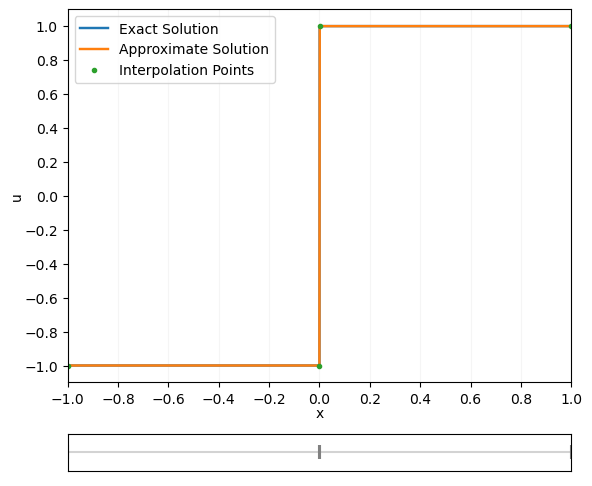}
\end{minipage}%
}
\subfigure[Approximation at $t=0.5$]{
\begin{minipage}[t]{.48\linewidth}
\includegraphics[scale=.47]{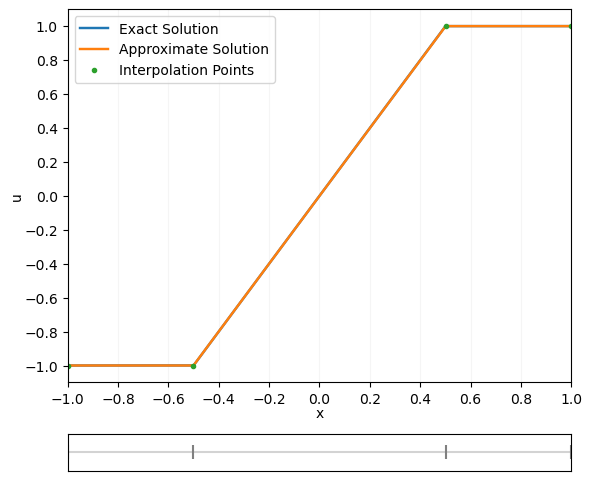}
\end{minipage}%
}%

\caption{Approximation results of the Riemann problem with rarefaction for Burgers' equation}
\end{figure}

\subsubsection{Riemann Problem with Rarefaction} The second test problem is the Riemann problem with the following initial data
\[
u(x,0) = \left\{
\begin{array}{rl}
-1, & x<0,\\[2mm]
1, & x \geq 0.
\end{array}
\right.
\]
The domain is $\Omega = (-1,1)$ and the time interval is $(0, 0.5)$. The exact solution is given by
\[
u(x,t) = \left\{
\begin{array}{rl}
-1, & x< -t,\\[2mm]
x/t, & -t\leq x \leq t,\\[2mm]
1, & x > t
\end{array}
\right.
\]
for $(x,t)$ in $\Omega \times I$.

The NN representation of the initial data was computed with the ANE method such that the relative numerical error in the $L^2$ norm was less than $\epsilon=3.0 \times 10^{-2}$. This resulted in a NN approximation with 4 neurons. Because the approximation to the initial data is continuous, there is a unique weak solution to (\ref{pde}) with the initial data $u_0(x) = u_N^{(0)}(x)$. Hence, characteristic propagation accurately computes the vanishing viscosity solution.


\subsubsection{Sinusoidal Initial Data}\label{burg-sine} The third test problem is equation (\ref{pde}) with the domain $\Omega = (0,1)$, the time interval $I = (0,1)$, and the following sinusoidal initial data
\[
u_0(x) = \sin(2 \pi x).
\]
The NN representation of the initial data was computed with the ANE method such that the relative numerical error in the $L^2$ norm was less than $\epsilon=10^{-3}$. As highlighted in table \ref{computation-1},  the ENN method required $587$ time steps and fewer than $78$ breaking points to compute the solution. Whereas, the WENO scheme for the reference solution took $2500$ time steps with $1000$ mesh points. 

This test problem demonstrates that the ENN method accurately approximates problems with changing shock heights (see Figure \ref{fig_burg_sine}).
As seen in Table \ref{tab_burg_sine} and Figure \ref{fig_burg_sine}, the number of breaking points decrease as the shock forms and propagates. After the shock forms at time $1/(2\pi)$, Table \ref{tab_burg_sine} indicates that the relative error jumps two orders of magnitude.


{
\renewcommand{\arraystretch}{1.1}
\begin{table}[!htb]\label{tab_burg_sine}
\setlength{\tabcolsep}{10pt}
\caption{Relative errors and the number of breaking points for Burgers' equation with sinusoidal initial data}
\vspace{6pt}
\begin{center}
\begin{tabular}{|c|c|c|c|}
\hline
Time & $\frac{\|\tilde{u}(\cdot, t_k)- u^{(k)}_{N}\|_{L^2(\Omega)}}{\|\tilde{u}(\cdot, t_k)\|_{L^2(\Omega)}}$  & $n_k$ \\ \hline 
0.0 & $6.6923\times 10^{-4}$ & 78\\ \hline
0.1 &  $7.8352\times 10^{-4}$ & 78\\ \hline
0.2  &  $4.0166\times 10^{-2}$ & 56\\ \hline
0.3 & $5.1491\times 10^{-2}$  & 38 \\ \hline
0.4  & $5.3515\times 10^{-2}$ & 30\\ \hline
0.5  & $5.4162\times 10^{-2}$ & 25\\ \hline
\end{tabular}
\end{center}
\end{table}
}

{
\renewcommand{\arraystretch}{1.2}
\begin{table}[!htb]\label{computation-1}
\setlength{\tabcolsep}{10pt}
\caption{Cost to compute the solution to Burgers' equation with sinusoidal initial data at the final time using the ENN method and the WENO scheme}
\vspace{6pt}
\begin{center}
\begin{tabular}{|c|c|c|}
\hline
Method & \# Mesh points & \# Time steps \\ \hline
ENN &  $<78$  & $587$  \\ \hline
WENO &  $1000$  & $2500$ \\ \hline
\end{tabular}
\end{center}
\end{table}
}


\begin{figure}[p]\label{fig_burg_sine}
\centering
\subfigure[Approximation at $t=0.0$]{
\begin{minipage}[t]{0.48\linewidth}
\includegraphics[scale = .47]{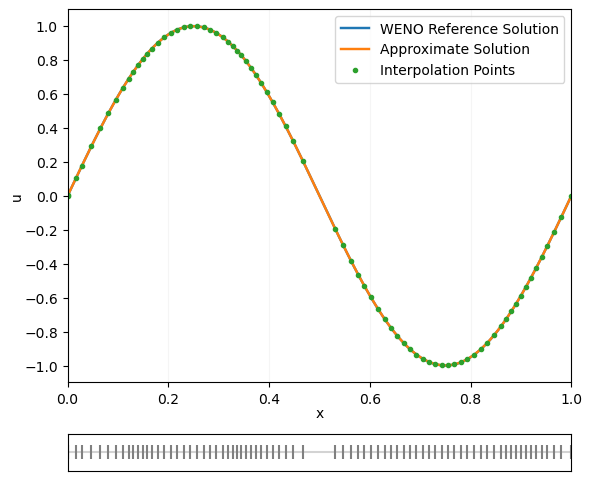}
\end{minipage}%
}
\subfigure[Approximation at $t=0.1$]{
\begin{minipage}[t]{.48\linewidth}
\includegraphics[scale=.47]{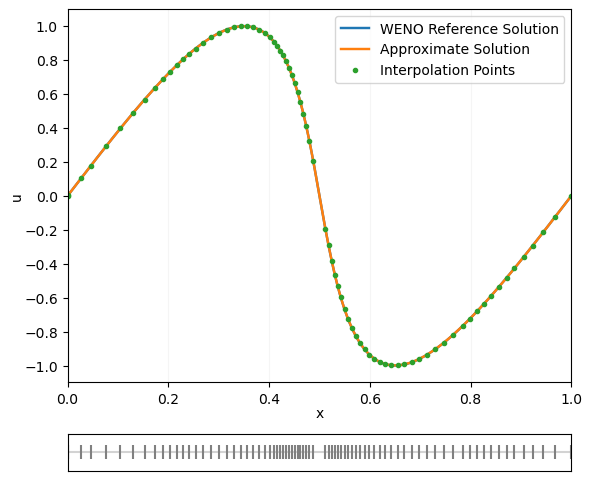}
\end{minipage}%
}%

\subfigure[Approximation at $t=0.2$]{
\begin{minipage}[t]{0.48\linewidth}
\includegraphics[scale = .47]{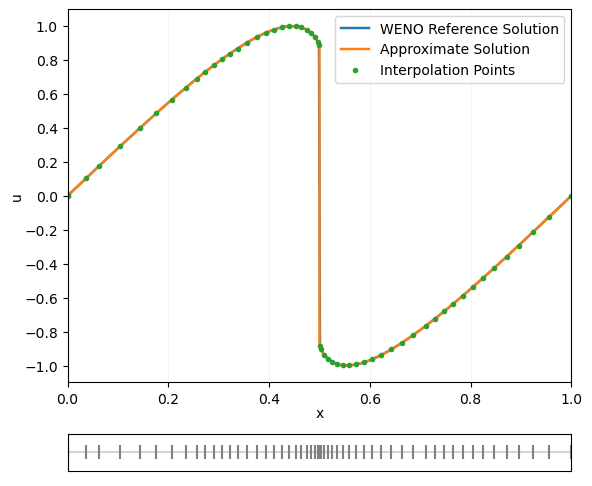}
\end{minipage}%
}
\subfigure[Approximation at $t=0.3$]{
\begin{minipage}[t]{.48\linewidth}
\includegraphics[scale=.47]{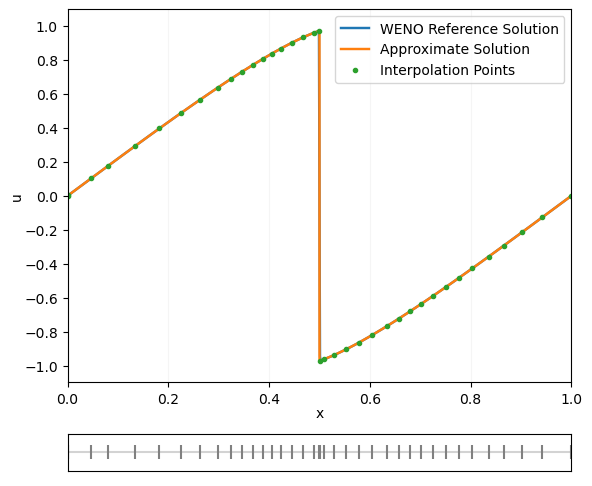}
\end{minipage}%
}%

\subfigure[Approximation at $t=0.4$]{
\begin{minipage}[t]{0.48\linewidth}
\includegraphics[scale = .47]{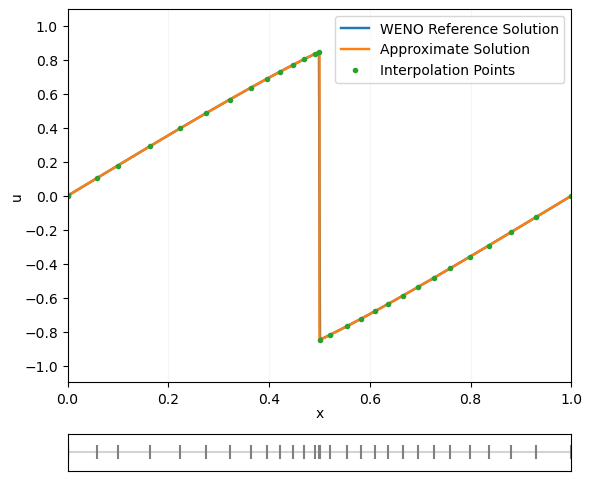}
\end{minipage}%
}
\subfigure[Approximation at $t=0.5$]{
\begin{minipage}[t]{.48\linewidth}
\includegraphics[scale=.47]{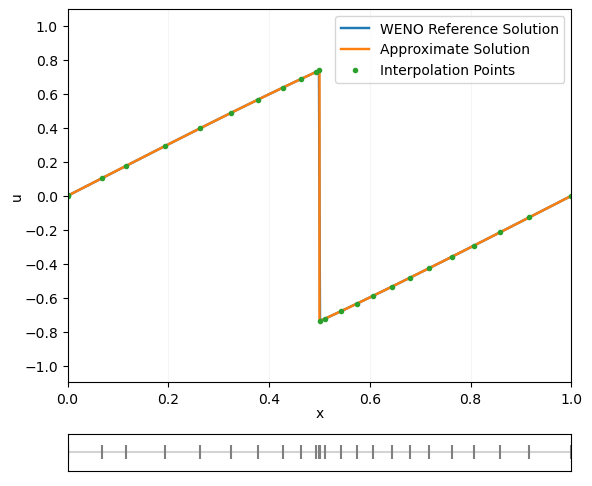}
\end{minipage}%
}%

\caption{Approximation results of Burgers' equation with sinusoidal initial data}
\end{figure}


\subsubsection{Exponential Initial Data}\label{burg-exp} The third test problem is equation (\ref{pde}) with the domain $\Omega = (-1,1)$, the time interval $I = (0,1)$, and the following initial data
\[
u_0(x) = \exp (-16x^2).
\]
The NN representation of the initial data was computed with the ANE method such that the relative numerical error in the $L^2$ norm was less than $\epsilon=10^{-3}$. 

As demonstrated in Table \ref{computation-2}, The ENN method approximates the solution accurately with $418$ time steps and fewer than $83$ breaking points. Whereas, the WENO scheme uses $5000$ time steps with $2000$ mesh points to compute the reference solution. As in the previous test problem, the number of breaking points decreases as the shock forms (see Table~\ref{tab_burg_bell} and Figure~\ref{fig_burg_bell}). After the shock forms at time $\sqrt{32e}/32$, Table~\ref{tab_burg_bell} indicates that the relative error jumps two order of magnitude.


{
\renewcommand{\arraystretch}{1.1}
\begin{table}[!htb]\label{tab_burg_bell}
\setlength{\tabcolsep}{10pt}
\caption{Relative errors and the number of breaking points for Burgers' equation with exponential initial data}
\vspace{6pt}
\begin{center}
\begin{tabular}{|c|c|c|c|}
\hline
Time & $\frac{\|\tilde{u}(\cdot, t_k)- u^{(k)}_{N}\|_{L^2(\Omega)}}{\|\tilde{u}(\cdot, t_k)\|_{L^2(\Omega)}}$ & $n_k$ \\ \hline 
0.0 & $6.6207\times 10^{-4}$ & 83 \\ \hline
0.2 &  $7.2902\times 10^{-4}$ & 83\\ \hline
0.4  &  $1.2718 \times 10^{-2}$ & 61 \\ \hline
0.6 & $2.1803\times 10^{-2}$ & 47 \\ \hline
0.8  & $2.0423\times 10^{-2}$ & 40 \\ \hline
1.0  & $1.4822\times 10^{-2}$ & 37\\ \hline
\end{tabular}
\end{center}
\end{table}
}

{
\renewcommand{\arraystretch}{1.2}
\begin{table}[!htb]\label{computation-2}
\setlength{\tabcolsep}{10pt}
\caption{Cost to compute the solution to Burgers' equation with exponential initial data at the final time using the ENN method and the WENO scheme}
\vspace{6pt}
\begin{center}
\begin{tabular}{|c|c|c|}
\hline
Method & \# Mesh points & \# Time steps \\ \hline
ENN &  $<83$  & $418$  \\ \hline
WENO &  $2000$  & $5000$ \\ \hline
\end{tabular}
\end{center}
\end{table}
}


\begin{figure}[p] \label{fig_burg_bell}
\centering
\subfigure[Approximation at $t=0.0$]{
\begin{minipage}[t]{0.48\linewidth}
\includegraphics[scale = .47]{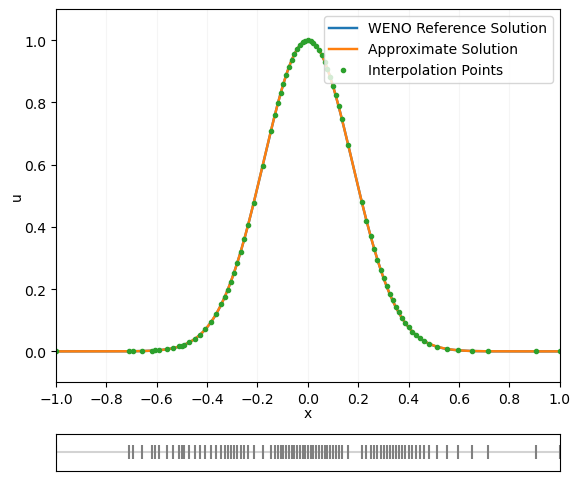}
\end{minipage}%
}
\subfigure[Approximation at $t=0.3$]{
\begin{minipage}[t]{.48\linewidth}
\includegraphics[scale=.47]{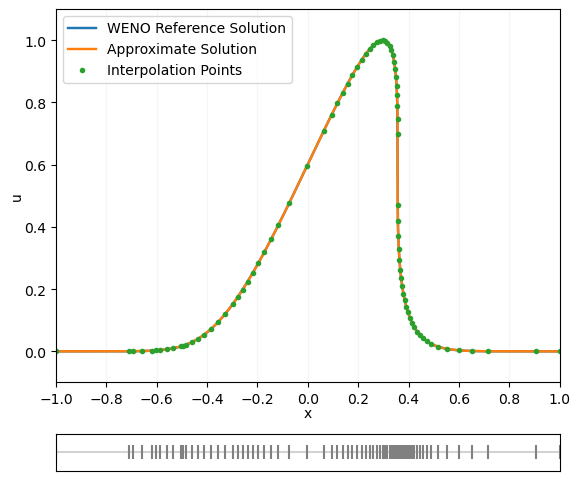}
\end{minipage}%
}%

\subfigure[Approximation at $t=0.4$]{
\begin{minipage}[t]{0.48\linewidth}
\includegraphics[scale = .47]{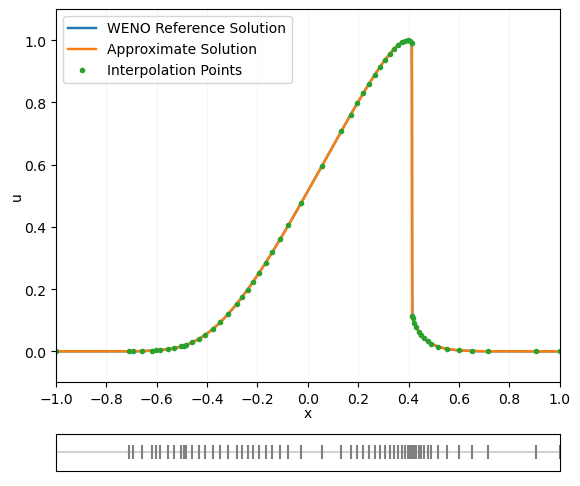}
\end{minipage}%
}
\subfigure[Approximation at $t=0.5$]{
\begin{minipage}[t]{.48\linewidth}
\includegraphics[scale=.47]{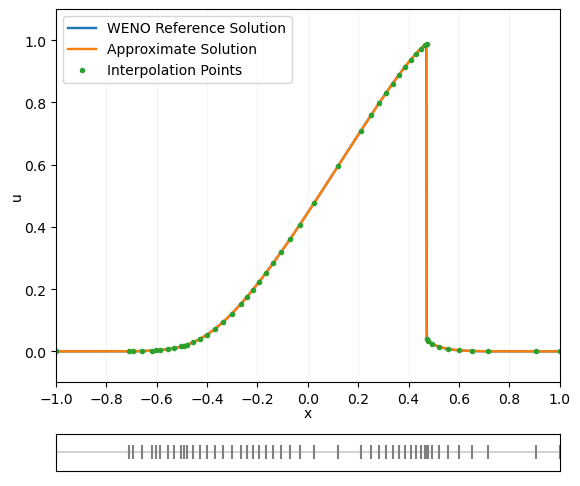}
\end{minipage}%
}%

\subfigure[Approximation at $t=0.6$]{
\begin{minipage}[t]{0.48\linewidth}
\includegraphics[scale = .47]{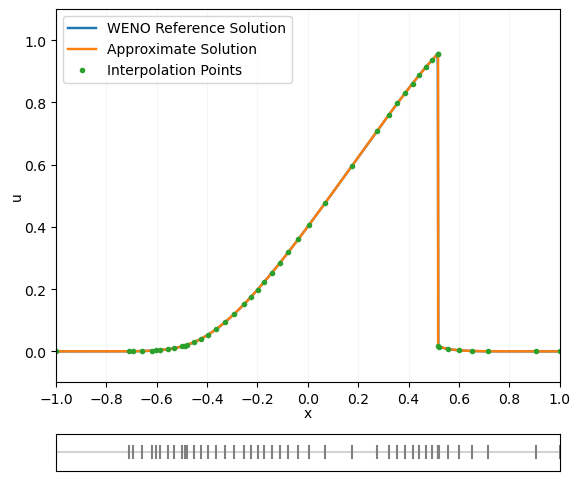}
\end{minipage}%
}
\subfigure[Approximation at $t=0.7$]{
\begin{minipage}[t]{.48\linewidth}
\includegraphics[scale=.47]{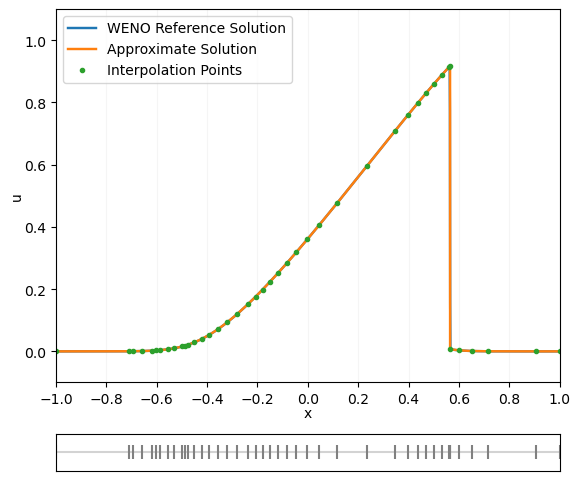}
\end{minipage}%
}%

\caption{Approximation results of Burgers' equation with exponential initial data}
\end{figure}

\section{Conclusions and Discussions}\label{conclusion}

The proposed ENN method in this paper emulates the physical phenomena of hyperbolic conservation laws by evolving a neural network representation of the initial data with the addition of boundary data. This evolution follows characteristic lines if there is no shock formation; otherwise, it is supplemented by the characteristic finite volume scheme introduced in Section~\ref{ENN-nonlinear} 
and proper time step control.

It is shown numerically that the ENN method can accurately simulate problems with contact discontinuities (e.g., the linear advection equation with discontinuous initial or boundary data) or a rarefaction wave (e.g., the Riemann problem for Burgers' equation) with ease. In this case, the ENN method only uses characteristic propagation. With the augmentation of automatic time step control and the characteristic finite volume scheme in shock regions, the ENN method can accurately replicate shock formation without common numerical artifacts such as oscillation, smearing, etc. This is demonstrated numerically for Burgers' equation with piecewise constant, sinusoidal, and exponential initial data. 

At each time step, the computational cost of the ENN method is proportional to the number of breaking points used for approximating the initial data. Hence, it is several orders of magnitude less than that of the existing explicit numerical schemes. The time step size is automatically chosen to capture the critical changes in the evolution of the solution. 
Specifically, if there is no shock formation, the step size has no restriction; otherwise, the time step size adapts to speed of shock formation.

Theoretically, it is proved that the ENN method at each time step preserves the accuracy of the initial and boundary data approximations for the linear advection problem. The error analysis for the inviscid Burgers equation will be presented in the companion paper \cite{ENN2}.

For general spatial fluxes and even for the linear advection problem with variable velocity, the initial profile of the underlying problem can deform dramatically as time evolves. If this phenomenon occurs, the ENN method in its current form 
is not able to approximate the solution accurately. This drawback of the ENN method will be investigated and resolved in a forthcoming paper.
\bibliographystyle{ieee}
\bibliography{ref}

\end{document}